\documentclass[11pt,reqno]{elsarticle}

\usepackage[title]{appendix}

\usepackage{amsfonts, graphicx, mathtools, amscd, amsmath, cancel}
\usepackage{xcolor}
\usepackage{epsfig}
\usepackage{subfigure}
\usepackage[english]{babel}
\usepackage{algorithm,algpseudocode} 
\usepackage{listings} 
\usepackage{multicol}
\usepackage[colorlinks=true, allcolors=blue]{hyperref}
\usepackage{tikz, amsmath}
\usetikzlibrary{shapes,arrows}

\providecommand{\U}[1]{\protect\rule{.1in}{.1in}}
\providecommand{\U}[1]{\protect\rule{.1in}{.1in}}
\allowdisplaybreaks
\newtheorem{theorem}{Theorem}
\newtheorem{lemma}{Lemma}
\newtheorem{proposition}{Proposition}

\newenvironment{proof}{\paragraph{Proof}}{\hfill$\square$}

\newtheorem{definition}{Definition}

\numberwithin{equation}{section}

\numberwithin{table}{section}
\numberwithin{figure}{section}
\usepackage{tikz}
\usetikzlibrary{arrows.meta,positioning}
\usetikzlibrary{arrows}
\usetikzlibrary{fit,shapes.geometric}

\usepackage{setspace}
\begin{document}



\begin{center}
{\Large
\textbf{Stochastic games of parental vaccination decision making and bounded rationality}
} \newline
\\
Andras Balogh \& Tamer Oraby$^{*}$
\\
School of Mathematical and Statistical Sciences, The University of Texas Rio Grande Valley, \\
1201 W. University Dr., Edinburg, Texas, USA.
\end{center}
$^*$Corresponding author: \\
 E-mail: tamer.oraby@utrgv.edu\\
 Phone: +1-956-665-3536

\section*{Abstract}
Vaccination is an effective strategy to prevent the spread of diseases. However, hesitancy and rejection of vaccines, particularly in childhood immunizations, pose challenges to vaccination efforts. In that case, according to rational decision-making and classical utility theory, parents weigh the costs of vaccination against the costs of not vaccinating their children. Social norms influence these parental decision-making outcomes, deviating their decisions from rationality. Additionally, variability in values of utilities stemming from stochasticity in parents' perceptions over time can lead to further deviations from rationality. In this paper, we employ independent white noises to represent stochastic fluctuations in parental perceptions of utility functions of the decisions over time, as well as in the disease transmission rates. This approach leads to a system of stochastic differential equations of a susceptible-infected-recovered (SIR) model coupled with a stochastic replicator equation. We explore the dynamics of these equations and identify new behaviors emerging from stochastic influences. Interestingly, incorporating stochasticity into the utility functions for vaccination and nonvaccination leads to a decision-making model that reflects the bounded rationality of humans. Noise, like social norms, is a two-sided sword that depends on the degree of bounded rationality of each group. We also perform a stochastic optimal control as a discount to the cost of vaccination to counteract bounded rationality.

{\bf keywords:} Replicator Dynamics, Stochastic Differential Equations, Game Theory, Disease Models.

\section{Introduction}\label{Sec:Intro}
In disease modeling, applying game theory to understand \emph{human behavior} about mitigation choices offers crucial insights into the dynamics of epidemics. Conventional methods often presuppose rational decision-making; however, as the COVID-19 pandemic has demonstrated, human behavior can sometimes hinder public health initiatives to control the spread of the disease. Individual reactions to pharmaceutical interventions such as vaccination and non-pharmaceutical measures, including mask use, are shaped by a complex mixture of factors that include social norms, misinformation, and trust in authorities \cite{bish2010demographic}.

Moreover, integrating concepts from behavioral economics and psychology, such as bounded rationality and prospect theory, adds realism to these models. These theories recognize that individuals do not always process information flawlessly or make decisions based solely on maximum utility \cite{tversky1974judgment}. Instead, decisions are influenced by limited cognitive resources, biases, and how information is framed.

Kermack and McKendrick (1927) initially laid the groundwork for deterministic models of epidemics, which have since evolved to include stochastic components \cite{kermack1927contribution}. In \cite{bailey1975mathematical}, a pioneer in stochastic disease modeling was introduced using stochastic processes like birth-and-death processes. In that study, the authors introduced stochastic elements into mathematical epidemiology, making it possible to model the inherent randomness of disease transmission and study extinction thresholds to better understand and predict the dynamics of infectious diseases. Several researchers have studied stochastic models of the spread of diseases in the form of stochastic differential equations (SDEs). SDEs incorporate randomness directly into the model, representing various sources of uncertainty and variability in disease transmission and progression. These models have been particularly beneficial in capturing the random nature of contact patterns and the influence of random events on the spread of infectious diseases. Incorporating SDEs into epidemic modeling provides a more comprehensive understanding of how diseases spread and how they can be controlled under uncertainty.  

Allen (1994) provided foundational work on stochastic epidemic models, demonstrating how SDEs can model demographic stochasticity and environmental variability in disease dynamics \cite{allen1994some}. Gray et al. (2011) extended this approach to the SIS model, exploring the existence of a stationary distribution and derived expressions for its mean and variance \cite{gray2011stochastic}. Tornatore et al. (2005) investigated the stability of a stochastic SIR system, providing insights into the conditions under which disease extinction or persistence occurs \cite{tornatore2005stability}. Furthermore, Gao et al. (2019) studied a delay differential equation SIS model with a general non-linear incidence rate, highlighting the impact of delays and non-linearities in disease transmission \cite{Gao2019}. Cao et al. (2017) focused on the dynamics of an epidemic model with vaccination in a noisy environment, emphasizing the role of stochastic effects in vaccination strategies \cite{Cao2017}. Liu (2023) modeled multiple transmission pathways in a scenario of waterborne pathogens using a stochastic model, illustrating the complexity and variability in disease transmission routes \cite{Liu2023}. See also \cite{allen2008introduction} for an introduction to stochastic epidemic models.

Researchers have continued to advance the field by exploring the effects of multiple noise sources on infectious disease models. For example, Jian et al. (2024) examined the impact of multisource noise on disease dynamics, offering insights into how different environmental randomness can influence disease spread \cite{Jian2024}. Babaei et al. (2023) applied SDEs to model the COVID-19 pandemic, integrating stochastic elements into SEIAQHR models to capture the variability and uncertainty of real-world scenarios \cite{Babaei2021}. Similarly, Iddrisu et al. (2023) highlighted the critical role of stochasticity in understanding cholera epidemiology, demonstrating how SDEs can provide more accurate predictions and inform public health interventions \cite{hindawi2023cholera}. These contributions underscore the importance of stochastic models in epidemiology, as they provide more realistic and flexible frameworks for predicting disease outcomes and formulating effective public health policies.

Rational decision-making theories often disregard information processing constraints, presuming that rational decision-makers always opt for the highest rewards without factoring in the effort or resources needed to determine the best action \cite{Fishburn1982,Savage1954, Von1944}. H.A. Simon's bounded rationality theory addresses these deviations, suggesting that decision-makers operate under limited information processing resources \cite{Simon1982, Simon1984}. Behavioral economics experiments reveal deviations between perfectly rational decision-makers and those constrained by bounded rationality \cite{Camerer2003}. 

\textit{Behavioral epidemiology} has evolved in the last two decades to become a major domain in the research of infectious diseases. The intricate interplay between human behavior and epidemiological outcomes highlights this. Foundational contributions to the literature of behavioral epidemiology, such as by \cite{bauch2005,donofrio2011,manfredi2013,reluga2006,wang2016}, have shaped the understanding of the impact of behavioral responses to disease outbreaks and their control. In particular, it gave insight into how parental vaccination decisions affect childhood disease dynamics and public health interventions, especially with vaccine scare \cite{bauch2012,oraby2,oraby1}. Free-riding behaviors \cite{ibuka2014}, social networks \cite{eames2009}, social learning \cite{ndeffo2012}, and economic incentives \cite{wagner2020}, to name a few factors, have been shown to impact vaccination uptake and thus disease spread. Those research findings emphasize the importance of integrating behavioral and epidemiological modeling approaches to address challenges to disease control \cite{wang2016}. They demonstrate the complex interplay of behavioral, social, cognitive, and economic factors in human behavior in general and vaccine decision-making, in particular, which form the basis for this research.

We seek to add a building block to behavioral epidemiology advances regarding factors influencing parental vaccination decisions, particularly rationality in the decision-making process. We extend the work in \cite{oraby1}, where the authors emphasized integrating social factors into behavioral responses and infectious disease models, and \cite{oraby2}, where the authors explored the use of prospect theory to enhance replicator dynamics in vaccination decision making. Here, we aim to build on these models by focusing on bounded rationality in decision-making resulting from stochasticity in the perception of the payoff of vaccination, leading to a mutation term in the replicator dynamics. Our results contribute to understanding parental behavior towards childhood immunization by showing how collective deviations in those perceptions can make disease persist. We also found that the consequences of these deviations could be counteracted using an optimal temporal discount in the cost of vaccination.

In this paper, we derive a stochastic differential equation of replicator dynamics (RD) for vaccination \cite{fudenberg1998theory}. The stochasticity arises from the perception of the utilities of the choices and deviations from rationality. Using stochastic game theory to model vaccine choice represents a significant step forward in capturing this complexity. In this framework, a collective decision to vaccinate is influenced not only by personal perceptions of risk-benefit but also by the choices made by others in their community and external but common environmental, cultural, informational, and economic factors. These shared factors are essential for understanding collective dynamics and population-level responses under systemic uncertainty. This interdependence of decisions, along with the inherent variability in the way individuals assess risks and benefits under uncertainty, can also be effectively modeled as a stochastic game in which the outcomes depend randomly on the actions of all participants \cite{bauch2004vaccination,betsch2013inviting, oraby1}. To our knowledge, we are the first to introduce and study a stochastic behavioral game theory model in the disease modeling literature. 

The paper is summarized as follows. In Section 2, we derive a new stochastic behavioral game theory equation that models vaccination versus non-vaccination choices. We pair that equation with a stochastic SIR model in which the transmission rate also evolves randomly. This will be followed in Section 3 by an analytical stability analysis of the model's equilibria. We also use stochastic simulations of stochastic differential equations to further study the stability of the model equilibria and corroborate some of the analytical results. Additionally, we use stochastic optimal control of the cost of vaccination to counteract the effect of bounded rationality. We follow that section with a discussion and conclusion section.

\section{Model}\label{Sec:Model}
Let $(\Omega,\mathcal{F},\{\mathcal{F}_t\}_{t\geq 0},\mathbb{P})$ be a complete probability space, with filtration $\{\mathcal{F}_t\}_{t\geq 0}$, with respect to which all Brownian motions $B_i$ and $W_i$ introduced in the following are defined. We use a stochastic replicator equation following the methods of establishing a stochastic replicator equation in \cite{benaim2008robust,Fudenberg1992} to guarantee the feasibility of solutions of the stochastic RD equation. The stochastic RD equation models the evolution over time of strategies or traits in a population by incorporating random fluctuations due to environmental variability or when the population is finite. For instance, let a population have two strategies or traits, A and B, like vaccination and non-vaccination strategies. If strategy A has a higher payoff than B, then the RD equation's deterministic component would imply that the adopters of strategy A must increase. However, the stochastic component could cause temporary increases in the adopters of strategy B, while strategy A is more gainful and vice versa.

Let $X(t)=(X_1(t),X_2(t))$ be the number of vaccinators and non-vaccinators in a community of a total size of $N:=X_1+X_2$ decision makers and let the fractions be denoted by $x_i=X_i/N$. Since $x_1+x_2=1$, we will also use the lower case $x$ to denote $x_1$ whenever it does not cause confusion. Let $u_i(x_1,x_2)=\sum_{j=1}^2 v(s_i,s_j)x_j$ be the expected payoff for a decision maker adopting strategy $s_i$ if randomly and uniformly matched with a decision maker adopting strategy $s_j$. The function $v(s_i,s_j)$ is the payoff for a decision maker who adopts strategy $s_i$ interacting with a decision maker adopting strategy $s_j$. 

A deviation from rationality due to external and internal factors is modeled here by adding a white noise term to the payoff $v(s_i,s_j)$ giving a new payoff function $v(s_i,s_j)+\nu_i \dot{B}_i(t)$. The stochastic process $\dot{B}(t)=(\dot{B}_1(t),\dot{B}_2(t))$ is a two-dimensional white noise. Volatility $\nu_i>0$ is the magnitude of the noise in the perceived utility. We assume here that all vaccinators and all non-vaccinators share the same volatility $\nu_1$ and $\nu_2$, respectively. Maintaining collective population dynamics, the stochastic RD model postulates that stochastic noise affects payoff perceptions uniformly across adopters of the same strategy. This does not account for individual variability in the perceived payoffs. That assumption reflects the influence of common external factors on risk perceptions, aligning with psychological theories of collective decision-making under uncertainty.

The following stochastic differential equation then gives the replicator equation modeling behavioral dynamics,
\begin{equation}\label{stocrep1}
    \dfrac{dX_i(t)}{dt}=\kappa X_i(t)\left(u_i\left(x_1,x_2\right)+\nu_i \dot{B}_i(t)\right)
\end{equation}
or 
\begin{equation}\label{stocrep11}
    dX_i(t)=\kappa X_i(t) u_i\left(x_1,x_2\right) dt+\kappa\nu_i X_i(t) dB_i(t)
\end{equation}
where $\kappa>0$ is the replication rate. Applying It\^o's formula (see Lemma \ref{ito} in the Appendix) to the transformation $x_1=X_1/N$, we get
\begin{equation}\label{stocrep2}
    dx_1=\kappa x_1 x_2 \left(v(s_1,s_2)-v(s_2,s_1)+\kappa \nu_2^2 x_2-\kappa \nu_1^2 x_1\right)dt+\kappa x_1 x_2 (\nu_1 dB_1(t)-\nu_2 dB_2(t))
\end{equation}
Since $\nu_1 dB_1(t)-\nu_2 dB_2(t)$ has the same distribution as $\sqrt{\nu_1^2+\nu_2^2}\, dW(t)$, where $dW(t)$ is also Gaussian with mean zero and variance $dt$, then equation \eqref{stocrep2} is equivalent to:
\begin{equation}\label{stocrep22}
    dx_1=\kappa x_1 x_2 \left(v(s_1,s_2)-v(s_2,s_1)+\kappa \nu_2^2 x_2-\kappa \nu_1^2 x_1\right)dt+\kappa x_1 x_2 \sqrt{\nu_1^2+\nu_2^2}\, dW(t)
\end{equation}
Equation \eqref{stocrep22} guarantees that $0\leq x_1(t)\leq 1$ for all $t>0$ if $0\leq x_1(0)\leq 1$.

We model the payoffs of vaccination to be $v(s_1,s_2)=-\omega+\delta\,x_1$ and of vaccination to be $v(s_2,s_1)=-I+\delta\,x_2$, see \cite{oraby1}. In terms of the vaccination choice model, $\kappa$ is the social learning rate, $\omega$ is the cost of vaccination, and $\delta$ is the group pressure. Social norms of vaccination and non-vaccination enforced by vaccinators and non-vaccinators add $\delta\,x_1$ and $\delta\,x_2$ to the payoffs of vaccination and non-vaccination strategies, respectively.

The noise in perceived payoff results in the introduction of a new \emph{mutation} term $\kappa^2 x_1 x_2 (\nu_2^2 x_2- \nu_1^2 x_1)$ to replicator dynamics; see, e.g. \cite{benaim2008robust,hofbauerevolutionary}. A mutation term induces irrational choices of a strategy that might not be optimal. In terms of vaccination, for certain large values of $\nu_1$, the mutation term is a bounded rational term that can cause agents to adopt the non-vaccination strategy while it has a smaller payoff; that is $v(s_2,s_1)<v(s_1,s_2)$. 

The transmission dynamics are assumed to follow a susceptible-infected-recovered (SIR) compartmental model. Putting $x_1=x$ and $x_2=1-x$ in equation \eqref{stocrep22} and coupling it with the SIR model yields the system of equations: 
\begin{equation}\label{SIR} 
\begin{aligned} 
\frac{dS}{dt} &= \mu\,(1-x)- \beta \, S\,I-\mu \, S-\sigma_1 S\,I\,\dot{W}_1   \\ 
\frac{dI}{dt} &= \beta \, S\,I-(\mu+\gamma) \, I +\sigma_1 S\,I\,\dot{W}_1 \\ 
\frac{dR}{dt} &= \mu \, x+\gamma \, I-\mu \, R \\ 
\frac{dx}{dt} &= \kappa\,x\,(1-x)\,\left[ -\omega+I+\delta\,(2\,x-1)+\kappa \left(\sigma_3^2 - (\sigma_2^2+\sigma_3^2) x \right) + \sqrt{\sigma_2^2+\sigma_3^2}\, \dot{W}_2  \right]
\end{aligned}   
\end{equation}
where $S$, $I$, and $R$ are the proportion of susceptible, infected, and recovered individuals in the population at time $t$, such that $S+I+R=1$. The parameter
$\mu$ is the per capita birth/death rate, $\beta$ is the transmission rate, $\gamma$ is the recovery rate. Recall that $\kappa$ is the social learning rate, $\omega$ is the cost of vaccination, and $\delta$ is the group pressure. The term $\mu(1-x)$ is the recruitment rate of non-vaccinated children to the susceptible compartment due to parental choice.

Let $\sigma_1$ be the magnitude of noise in disease transmission. The last equation in \eqref{SIR}, a stochastic replicator equation, follows the replicator model introduced by \cite{Fudenberg1992}, where it is shown to have solutions $x$ between zero and one; see above. Notice that we changed the notation $\nu_1$ to $\sigma_2$ and $\nu_2$ to $\sigma_3$, to be in order within the coupled model. If $\sigma_2^2=\sigma_3^2=\frac{\delta}{\kappa}$ then, $ \delta\,(2\,x-1)+\kappa \left(\sigma_3^2 - (\sigma_2^2+\sigma_3^2) x \right)=0$. In this case, the collective noise would cancel out the group pressure in the stochastic replicator equation. Hence, it seems that bounded rationality could nullify or counteract social norms.

Since $S+I+R=1$, then the third equation of $R$ in equation \eqref{SIR} is redundant, so henceforth we will consider the model \eqref{SIR} without it and stop using the recovered compartment. Define the set of solutions for the model in Equation \eqref{SIR} as $\mathcal{S}=\{(s,i,x)\in \mathbb{R}_+^3: 0\leq s+i \leq 1 \text{ and } 0\leq x \leq 1\}$. In the following subsections, let $(S(t),I(t),x(t))$ be a solution of the model \eqref{SIR} with initial state $(S(0),I(0),x(0))\in \mathcal{S}$. The proofs of existence, uniqueness, and boundedness of the solution of the model \eqref{SIR} are given in \ref{A.ext}.

\section{Model's Equilibria and their Stability}
In this section, we discuss the different equilibria of the model in \eqref{SIR} and their stability. We note that some specific states of the three stochastic processes $S$, $I$, and $x$ are absorbing states. The states $0$ and $1$ are absorbing states for $x$ and $0$ is an absorbing state for $I$. 

We will support the results in this section with simulations using the Milstein algorithm for the numerical simulation of SDEs \cite{kloeden1992} implemented using Python. We will use parameter values from \cite{oraby1}, such as $\mu=1/50$ year$^{-1}$, $\gamma=365/22$ year$^{-1}$, and $\kappa=1.69$ year$^{-1}$. We have to mention here that most of the simulated processes exhibit slow mixing rates, and hence they take a very long time to converge to equilibrium.

\subsection{Model's equilibria} 
The model has five equilibria of the drift term (the deterministic part) in equation \eqref{SIR}. Those five equilibria and their existence conditions are found to be those of \cite{oraby1} when $\sigma_1^2=\sigma_2^2=\sigma_3^2=0$. Three of those are disease-free equilibria: 
\begin{enumerate}
    \item full vaccine uptake and no susceptibility,  
$
\mathcal{E}_1 \equiv (S_1,I_1,x_1)=(0,0,1) \, , 
$

\item no vaccine uptake and full susceptibility, 
$
\mathcal{E}_2=(1,0,0)
$, and 
\item partial vaccine uptake and partial susceptibility, 
$
\mathcal{E}_3 \equiv (1-x_3,0,x_3)
$, where 
\begin{equation}
    x_3=\frac{\kappa \sigma_3^2-\delta-\omega}{\kappa\sigma_2^2+\kappa\sigma_3^2-2\delta}. \label{x3}
\end{equation}

The equilibrium $\mathcal{E}_3$ exists in two regions:  $$\mathcal{R}_{3,1}=\left\{(\delta,\omega)\in \mathbb{R}^2_+:\delta-\kappa \sigma_2^2<\omega <\kappa \sigma_3^2-\delta  \right\}$$ and $$\mathcal{R}_{3,2}=\left\{(\delta,\omega)\in \mathbb{R}^2_+: \kappa\sigma_{3}^{2}-\delta<\omega<\delta-\kappa\sigma_{2}^{2} \right\}.$$
The two boundaries of either region intersect at $\left(\frac{\kappa}{2}(\sigma_2^2+\sigma_3^2), \frac{\kappa}{2}(\sigma_3^2-\sigma_2^2) \right)$. 
\end{enumerate}
The other two are endemic equilibria that exist when the basic reproduction number $\displaystyle R_0 = \dfrac{\beta}{\mu+\gamma} > 1$. They are:  
\begin{enumerate}
    \item no vaccine coverage, 
$
\displaystyle \mathcal{E}_4 \equiv \left(\frac{1}{R_0},\frac{\mu}{\mu+\gamma}\left(1-\frac{1}{R_0}\right),0\right) 
$, and

\item partial vaccine coverage, 
$
\displaystyle \mathcal{E}_5 \equiv \left(\frac{1}{R_0},\frac{\mu}{\mu+\gamma}\left(1-\frac{1}{R_0}-x_5\right),x_5\right), 
$
where 
$$x_5=\dfrac{\mu\left(1-\dfrac{1}{R_0}\right)+\left(\kappa\sigma_3^2-\delta-\omega\right)(\mu+\gamma)}{\mu+\left(\kappa\sigma_2^2+\kappa\sigma_3^2-2\delta\right)(\mu+\gamma)}.  $$
The equilibrium $\mathcal{E}_5$ exists in two regions $$\mathcal{R}_{5,1}=\left\{(\delta,\omega)\in \mathbb{R}^2_+:-\kappa \sigma_2^2+ \frac{1}{R_0} \kappa (\sigma_2^2+\sigma_3^2)+\delta(1-\frac{2}{R_0}) <\omega<-\delta+\frac{\mu}{\mu+\gamma}\,(1-\frac{1}{R_0})+\kappa \sigma_3^2\right\}$$ and 
$$\mathcal{R}_{5,2}=\left\{(\delta,\omega)\in \mathbb{R}^2_+:-\kappa \sigma_2^2+ \frac{1}{R_0} \kappa (\sigma_2^2+\sigma_3^2)+\delta(1-\frac{2}{R_0}) >\omega> -\delta+\frac{\mu}{\mu+\gamma}\,(1-\frac{1}{R_0})+\kappa \sigma_3^2\right\}$$
The two boundaries of either region intersect at $\left(\dfrac12 \dfrac{\mu}{\mu+\gamma}+\frac{\kappa}{2}(\sigma_2^2+\sigma_3^2), \dfrac12 \dfrac{\mu}{\mu+\gamma} (1-\frac{2}{R_0})+\frac{\kappa}{2}(\sigma_3^2-\sigma_2^2) \right)$. 
\end{enumerate}

The disease-endemic equilibria appear in the next stability analyses as types of behavior and not as exact values. The equilibria $\mathcal{E}_3$ and $\mathcal{E}_5$ represent the partial acceptance of the vaccine. However, in the latter, partial acceptance of the vaccine does not lead to the eradication of the disease.

\subsubsection{Dynamical regimes of stability} 
Define the mean-value function $\overline{Y}(T)$ of any function $Y(t)$ over $[0,T]$ to be
$$\overline{Y}(T)=\frac{1}{T} \int_0^T Y(t) dt$$ and let its limit be defined as $$Y_0=\lim_{T\to \infty} \overline{Y}(T).$$
Let $\displaystyle Y_*=\liminf_{t\to \infty} Y(t)$ and $\displaystyle Y^*=\limsup_{t\to \infty} Y(t)$. 

Next, we define modes of stability, such as exponential stability, and introduce a new mode of stability that we call logistic stability. Logistic stability helps in studying the impermanence of dynamic processes that take value in the unit interval $[0,1]$. To our knowledge, this concept has not been used before in dynamical system analyses.

\begin{definition}[Almost Sure Exponential Stability]
    We say that a process $Y(t)\in [0,\infty)$ has an almost sure exponentially stable equilibrium at $y=0$ if for some $c>0$,
    $$\limsup_{t\to \infty} \frac{1}{t} \log Y(t) \leq -c$$ almost surely.
\end{definition}

\begin{definition}[Almost Sure Logistic Stability]
    We say that a process $Y(t)\in [0,1]$ has an almost sure logistically stable equilibrium at $y=0$ and $1$ if for some $c>0$,
    $$\lim_{t\to \infty} \frac{1}{t} \log \frac{Y(t)}{1-Y(t)} = (-1)^{y+1} c$$ almost surely.
\end{definition}

The unit interval represents a natural domain for many real-life phenomena, such as proportions, probabilities, or normalized populations constrained to the unit interval $[0,1]$. The concept of logistic stability serves the need to characterize the transient dynamics of those types of processes as well as their behavior in convergence towards stability. Other convergence behaviors like exponential or asymptotic stability provide some types of convergence to equilibria. Those types may not capture other systems where the dynamics of the processes are bound and/or demonstrate logistic-like growth or decay which tend to be slower than the exponential ones. In infectious disease modeling, logistic stability could capture the impermanence of disease states within the unit interval, like prevalence, when epidemics go through transient phenomena. Logistic stability could distinguish those transient dynamics from other forms, emphasizing steady-state behavior. The limiting behavior of $\frac{1}{t} \log \frac{Y(t)}{1-Y(t)}$  can quantify the rate at which a disease dies out or persists, providing a new metric for control measures.

Define the stochastic basic reproduction number
\begin{equation}\label{R0s}
    R_0^s:=\frac{\beta}{\mu+\gamma+\frac12 \sigma_1^2}.
\end{equation}
which is always smaller than or equal to $R_0$.

First, we will study the global stability of a disease-free equilibrium. There are several cases in which the disease-free equilibrium is stable. The first case depends on the stochastic basic reproduction number and the magnitude of the noise in transmission.

\begin{theorem}\label{thm:diseasefree}
If $R_0^s<1$ and $\displaystyle \sigma_1^2\leq  \beta$ or if $\displaystyle \frac{\sigma_1^2}{\beta}>\max\left(1,\dfrac{R_0}{2}\right)$, then a disease-free equilibrium is exponentially globally stable almost surely.
\end{theorem}

\begin{proof}
We will prove the theorem in two parts, by showing that the following two inequalities hold: 
\begin{enumerate}
    \item[(C.I)] If $\displaystyle \frac{\sigma_1^2}{\beta}>\max\left(1,\dfrac{R_0}{2}\right)$, then 
    \begin{equation}\label{limsup1}
        \limsup_{t\to \infty} \dfrac{\log I(t)}{t}\leq -\dfrac{\beta(\mu+\gamma)}{\sigma_1^2} \left(\frac{\sigma_1^2}{\beta}-\dfrac{R_0}{2}\right) <0 \text{ a.s.} 
    \end{equation}
    \item[(C.II)] If $R_0^s<1$ and $\sigma_1^2 \leq \beta$, then 
    \begin{equation}\label{limsup2}
        \limsup_{t\to \infty} \dfrac{\log I(t)}{t}\leq -\left(\mu+\gamma+\frac12 \sigma_1^2\right)\left(1-R_0^s\right)  <0 \text{ a.s.}
    \end{equation}
\end{enumerate}
Let $V(t,I)=\log(I)$, then by It\^o's formula (see Lemma \ref{ito} in the Appendix),
\begin{align*}
dV=d\log(I)&=V_t dt +V_I (\beta \, S\,I-(\mu+\gamma) \, I) dt + \frac12 (\sigma_1 S\,I)^2 V_{II} dt + (\sigma_1 S\,I) V_I dW_1(t) \\
&= (\beta \, S-(\mu+\gamma) -\frac12 \sigma_1^2 S^2) dt+\sigma_1 S  dW_1(t)
\end{align*}
which by integration over the interval $(0,t)$ gives
\begin{align*}
\log(I(t))-\log(I(0))&= \int_0^t (\beta \, S(u)-(\mu+\gamma) -\frac12 \sigma_1^2 \left(S(u))^2\right) du+\sigma_1 \int_0^t S(u)  dW_1(u)\\
&= \int_0^t \beta \, S(u) \left( 1-\frac{\sigma_1^2}{2\beta} S(u)\right) du-(\mu+\gamma)t+M(t)
\end{align*}
where $M(t)=\sigma_1 \int_0^t S(u)  dW_1(u)$ is a local continuous martingale with $M(0)=0$. Since
\begin{equation}
    \limsup_{t\to \infty} \dfrac{E\left(M(t)^2\right)}{t} = \limsup_{t\to \infty} \dfrac{\sigma_1^2 \int_0^t (S(u))^2 du}{t}  
    \leq \limsup_{t\to \infty} \dfrac{\sigma_1^2 t}{t} 
    = \sigma_1^2<\infty,
\end{equation}
then by the Strong Law of Large Numbers, $\displaystyle \limsup_{t\to \infty} \dfrac{M(t)}{t}=0$ almost surely, , see \cite{mao2007stochastic} and Lemma \ref{martin} in the Appendix.

Now to prove (C.I), when $\displaystyle \frac{\sigma_1^2}{\beta}>\max\left(1,\dfrac{R_0}{2}\right)$, note that the quadratic function 
\begin{equation}
     f(S)=\beta S\left(1-\frac{\sigma_1^2}{2\beta} S\right) = \frac{\sigma_1^2}{2} S\left(\frac{2\beta}{\sigma_1^2}- S\right)
\end{equation} 
attains its maximum value of $\displaystyle \frac{\beta^2}{2\sigma_1^2}$ when $\displaystyle S=\frac{\beta}{\sigma_1^2}$, which is in the interval $(0,1)$ if $\sigma_1^2> \beta$, in which case  
\begin{align*}
\frac{\log(I(t))-\log(I(0))}{t}
&= 
\frac{\int_0^t \beta \, S(u) \left( 1-\frac{\sigma_1^2}{2\beta} S(u)\right) du-(\mu+\gamma)t+M(t)}{t} \\
& \leq \frac{\beta^2}{2\sigma_1^2}-(\mu+\gamma)+\dfrac{M(t)}{t}
\end{align*}
and so
\begin{equation*}
\limsup_{t \to \infty}\dfrac{\log(I(t))-\log(I(0))}{t}
\leq \frac{\beta^2}{2\sigma_1^2}-(\mu+\gamma) \text{ a.s.}
\end{equation*}
From this  (C.I) follows as
    \begin{align*}
        \limsup_{t\to \infty} \frac{\log (I(t))}{t} 
        &\leq  \frac{\beta^2}{2\sigma_1^2}-(\mu+\gamma) \\
        & =  -\frac{\beta(\mu+\gamma)}{\sigma_1^2} \left(\frac{\sigma_1^2}{\beta}-\frac{\beta}{2(\mu+\gamma)}\right) \\
        & = -\frac{\beta(\mu+\gamma)}{\sigma_1^2} \left(\frac{\sigma_1^2}{\beta}-\frac{R_0}{2}\right) \\
        & <0 \text{ a.s.} 
    \end{align*}

To prove (C.II), when $\sigma_1^2 \leq \beta$, consider the function 
\begin{equation}
f(S)=\beta S\left(1-\frac{\sigma_1^2}{2\beta} S\right)  = \frac{\sigma_1^2}{2}S\left(\frac{2\beta}{\sigma_1^2}-S \right),
\end{equation}
which attains a maximum value on the  interval $[0,1]$ at $S=1$ with maximum value
\begin{equation}
\max_{S\in [0,1]}f(S) =  \beta -\frac{\sigma_1^2}{2}.    
\end{equation}
 Thus,  
\begin{align*}
\dfrac{\log(I(t))-\log(I(0))}{t}
& = \frac{\int_0^t \beta \, S(u) \left( 1-\frac{\sigma_1^2}{2\beta} S(u)\right) du-(\mu+\gamma)t+M(t)}{t} \\
& \leq \beta -\frac12 \sigma_1^2 -(\mu+\gamma)+\dfrac{M(t)}{t}
\end{align*}
and so 
\begin{equation*}
\limsup_{t \to \infty}\dfrac{\log(I(t))-\log(I(0))}{t}
\leq \beta -\frac12 \sigma_1^2 -(\mu+\gamma) \text{ a.s.}
\end{equation*}
From this  (C.II) follows as
 \begin{align*}
        \limsup_{t\to \infty} \frac{\log (I(t))}{t} 
        &\leq  \beta -\frac12 \sigma_1^2 -(\mu+\gamma)  \\
        & = -\left(\mu+\gamma+\frac12 \sigma_1^2\right)\left(1-\frac{\beta}{\mu+\gamma+\frac12 \sigma_1^2}\right) \\
       & = -\left(\mu+\gamma+\frac12 \sigma_1^2\right)\left(1-R_0^s\right)  \\
        & <0 \text{ a.s.}
    \end{align*}
\end{proof}

Theorem \ref{thm:diseasefree} seems to be a classic conclusion; see, e.g., \cite{han2015threshold}; but it is shown here to hold even with the presence of a stochastic process $x(t)$ in the first equation. It shows that disease extinction could still be achieved here even when $R_0>1$ as long as $R_0^s<1$, see Figure \ref{fig:fig1} (a). Another case in which the disease could be eradicated while $R_0>1$ is when $\displaystyle \frac{\sigma_1^2}{\beta}>\max\left(1,\dfrac{R_0}{2}\right)$ as shown in Figure \ref{fig:fig1} (b). Even when $R_0^s>1$; or specifically, when the conditions of Theorem \ref{thm:diseasefree} are not satisfied, eradication of the disease is possible through vaccination, as we will see next.   

\begin{figure}[H]
    \centering
     \subfigure[]{\includegraphics[width=8.5 cm]{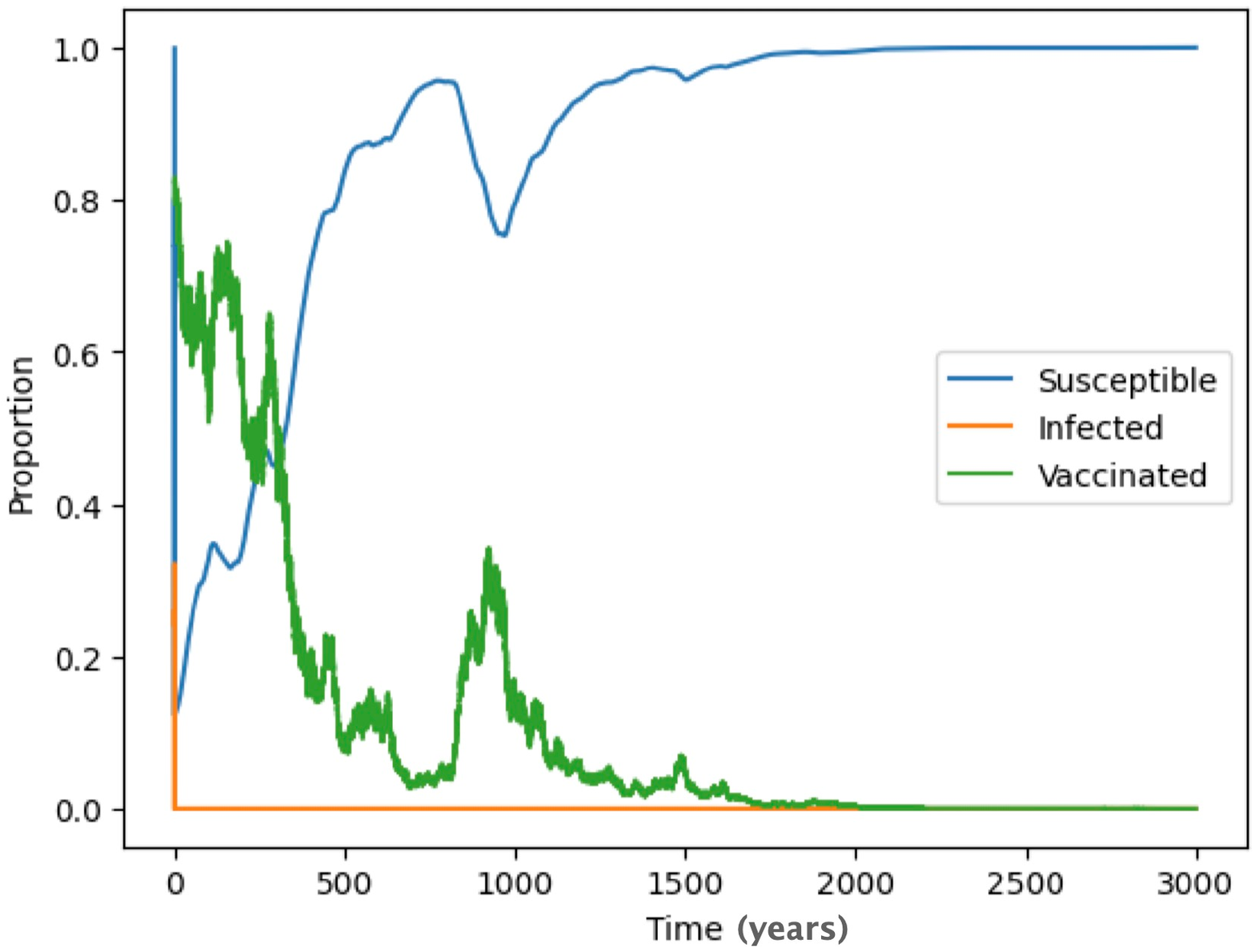}}
     \subfigure[]{\includegraphics[width=8.5 cm]{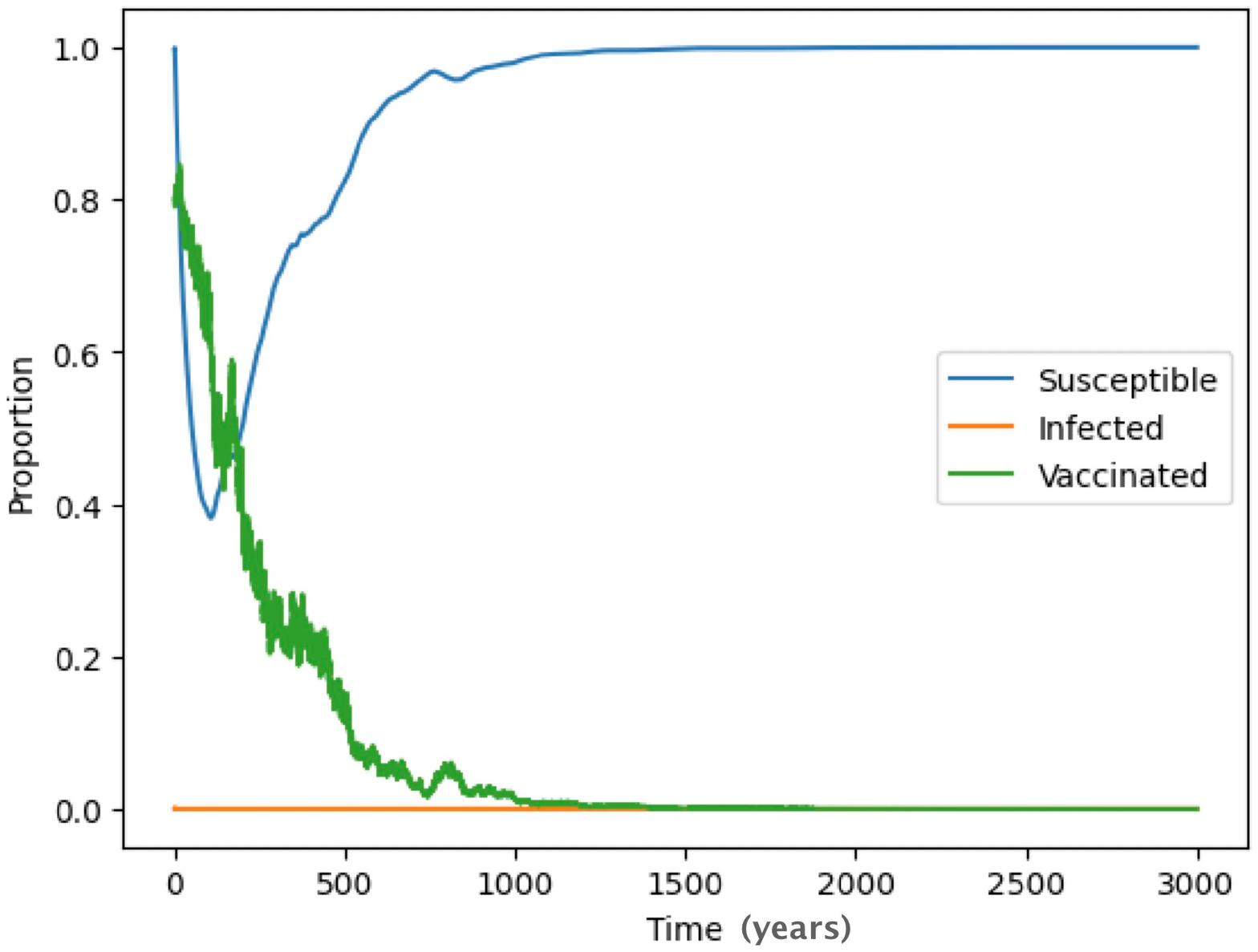}}   
    \caption{Simulation of susceptible $S$, infected $I$ and vaccinated $x$, when (a) $R_0^s<1$ and $\displaystyle \sigma_1^2\leq  \beta$ at $\beta=31$ and $\sigma_1^2=30$  giving $R_0=1.87$ and $R_0^s=.98$, and (b) $\displaystyle \frac{\sigma_1^2}{\beta}>\max\left(1,\dfrac{R_0}{2}\right)$ at $\beta=33.3$ and $\sigma_1^2=50$ giving $R_0=2$ and $R_0^s=.8$ (even when it is more than one). The rest of the parameters are $\mu=1/50$ year$^{-1}$, $\gamma=365/22$ year$^{-1}$, $\kappa=1.69$ year$^{-1}$, $\omega=.0015$, $\delta=.0005$, $\sigma_2^2=.0008$, and $\sigma_3^2=.0006$.}
    \label{fig:fig1}
\end{figure}

In \cite{oraby1}, using a deterministic model, it was shown that while $R_0>1$, the disease could be eradicated as the uptake of the vaccine increases to the level of $100\%$ when the peer pressure of the vaccination group is greater than the cost of the vaccine ($\delta>\omega$). However, it was also shown that if the cost of vaccination and the peer pressure of the non-vaccinating group exceed the risk of the disease ($\omega+\delta>\frac{\mu}{\mu+\gamma}(1-\frac{1}{R_0})$), parents might skip vaccination and the disease becomes endemic; see also \cite{oraby2}. Moreover, both equilibria were shown to be bistable. Here, bounded rationality changes those findings in a way that noise can impede the roles played by the cost of vaccination and peer pressure in \cite{oraby1}. The following theorem shows how noises in the perceived utilities can change the vaccination selection, and consequently the disease dynamics. 

\begin{theorem}\label{thm:vax}
If $\bar{I}(t)\xrightarrow[]{t\to\infty} I_0$, then $x(t)$ has an almost sure logistic stable equilibrium at $0$ and $1$ under the following conditions:
\begin{enumerate}
    \item\label{vax: 1} If $\omega+\kappa\frac{\sigma_2^2-\sigma_3^2}{2}>\max(I_0-\delta,\delta)$, then $x(t)\xrightarrow[]{t\to\infty} 0$ logistically almost surely, and
    \item\label{vax: 2} If $\omega+\kappa\frac{\sigma_2^2-\sigma_3^2}{2}<\min(I_0-\delta,\delta)$, then $x(t)\xrightarrow[]{t\to\infty} 1$ logistically almost surely, and also, $I(t)\xrightarrow[]{t\to\infty} 0$ and $S(t)\xrightarrow[]{t\to\infty} 0$ $a.s.$ whether $R_0^s$ is more or less than one.
\end{enumerate}
\end{theorem}
\begin{proof}
Using It\^o's formula (see Lemma \ref{ito} in the Appendix),
\begin{eqnarray*}
d\log(x(t))&=&\kappa (1-x(t)) \left[(\kappa\sigma_3^2-\delta-\omega)- (\kappa\sigma_2^2+\kappa\sigma_3^2-2\delta)x(t) +I(t)-\kappa \frac{\sigma_2^2+\sigma_3^2}{2} (1-x(t))\right]dt \\ && + \kappa \sqrt{\sigma_2^2+\sigma_3^2}  (1-x(t)) d{W}_2(t)
\end{eqnarray*}
and similarly
 \begin{eqnarray*}
d\log(1-x(t))&=&-\kappa x(t) \left[(\kappa\sigma_3^2-\delta-\omega)- (\kappa\sigma_2^2+\kappa\sigma_3^2-2\delta)x(t) +I(t)+\kappa \frac{\sigma_2^2+\sigma_3^2}{2} x(t)\right]dt \\ && -\kappa \sqrt{\sigma_2^2+\sigma_3^2}  x(t) d{W}_2(t).
\end{eqnarray*}

Therefore,
\begin{eqnarray*}
d\log \left(\dfrac{x(t)}{1-x(t)} \right)&=& \kappa\left(-\kappa \frac{\sigma_2^2-\sigma_3^2}{2}-\delta-\omega+2\delta x(t) +I(t)\right)dt  +\kappa \sqrt{\sigma_2^2+\sigma_3^2}   d{W}_2(t).
\end{eqnarray*}
Since $\frac{W_2(t)}{t}\xrightarrow[]{t\to\infty} 0$ almost surely due to the Strong Law of Large Numbers of Brownian motion, then 
\begin{equation}
\lim_{t \to \infty} \frac{\log \left(\dfrac{x(t)}{1-x(t)}\right)}{t}
= \kappa L(x_0,I_0)
\end{equation}
almost surely, where
$$L(x_0,I_0):=-\kappa \frac{\sigma_2^2-\sigma_3^2}{2}-\delta-\omega+2\delta x_0 +I_0.$$ 

If $x_0=0$, which is equivalent to $x(t)$ approaches $0$ almost surely, then $L(x_0,I_0)=-\kappa \frac{\sigma_2^2-\sigma_3^2}{2}-\omega-\delta+I_0$. If $x_0=1$ and $I_0=0$, which is equivalent to $x(t)$ approaches $1$ almost surely, then $L(x_0,I_0)=-\kappa \frac{\sigma_2^2-\sigma_3^2}{2}-\omega+\delta$. The first and second values of $L(x_0,I_0)$ could have opposite signs for the same parameter values as when $I_0-\delta<\kappa \frac{\sigma_2^2-\sigma_3^2}{2}+\omega<\delta$. Therefore, $x(t)$ approaches $0$ logistically almost surely if $-\kappa \frac{\sigma_2^2-\sigma_3^2}{2}-\omega-\delta+I_0<0$ and $-\kappa \frac{\sigma_2^2-\sigma_3^2}{2}-\omega+\delta<0$. Meanwhile, $x(t)$ approaches $1$ logistically almost surely if $-\kappa \frac{\sigma_2^2-\sigma_3^2}{2}-\omega+\delta>0$ and $-\kappa \frac{\sigma_2^2-\sigma_3^2}{2}-\omega-\delta+I_0>0$.
\end{proof}

An overlap region appears in Figure \ref{fig:fig2a} (a) as stated in the proof of Theorem \ref{thm:vax} when $I_0-\delta<\kappa \frac{\sigma_2^2-\sigma_3^2}{2}+\omega<\delta$. In that region, both equilibria could be stable, as the simulation demonstrates below. The dashed region with vertical lines in Figure \ref{fig:fig2a} (a) is for $\omega+\kappa\frac{\sigma_2^2-\sigma_3^2}{2}>I_0-\delta$ in which the disease-endemic and no-vaccination equilibrium $\mathcal{E}_4$-type equilibrium is stable; see Figure \ref{fig:fig2a} (b) for a simulation of a case in that region. Meanwhile, the dashed region with diagonal lines in Figure \ref{fig:fig2a} (a) is for $\omega+\kappa\frac{\sigma_2^2-\sigma_3^2}{2}<\delta$, in which disease-free and full-vaccination equilibrium $\mathcal{E}_1$ is stable; see Figure \ref{fig:fig2a} (c) for a simulation of a case in that region. Again, in the common region, both $\mathcal{E}_1$ and $\mathcal{E}_4$ are stable, see below.

When $\sigma_2^2=\sigma_3^2=0$, the two conditions of Theorem \ref{thm:vax} are shown to be sufficient for the local asymptotic stability of $\mathcal{E}_4$ and $\mathcal{E}_1$, respectively, for the deterministic model in \cite{oraby1}. In particular, the latter equilibrium of the disease-free state was shown to be stable when $\delta>\omega$. However, the magnitudes of the white noise in the utilities of both groups, $\sigma_2^2$ and $\sigma_3^2$, influence the opinion and uptake of vaccination. In Figure \ref{fig:fig2a} (b), the disease cannot be eradicated even when $\delta>\omega$ is mainly due to the level of noise in the vaccinator group. Figure \ref{fig:fig2a} (c) represents a case where the conditions of Theorem \ref{thm:diseasefree} are not satisfied, but disease eradication is still possible as implied by part \ref{vax: 2} of Theorem \ref{thm:vax}. If $R_0^s>1$, the disease could be eradicated by increasing the peer pressure of the vaccination group or decreasing the noise in the utility of the vaccinator's group conducing to $\omega+\kappa\frac{\sigma_2^2-\sigma_3^2}{2}<\delta$. Notice also that an increase in the noise of the non-vaccinator group could also increase vaccine uptake.   

\begin{figure}[H]
    \centering
     \subfigure[]{\includegraphics[width=8.5 cm]{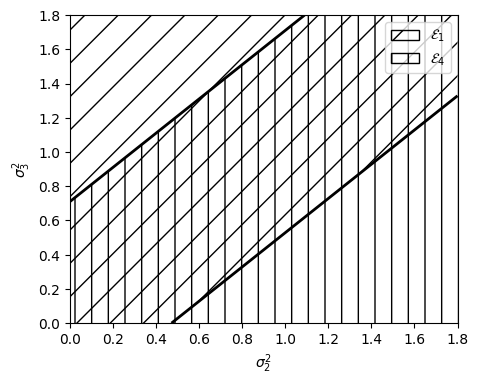}}
     \subfigure[]{\includegraphics[width=8.5 cm]{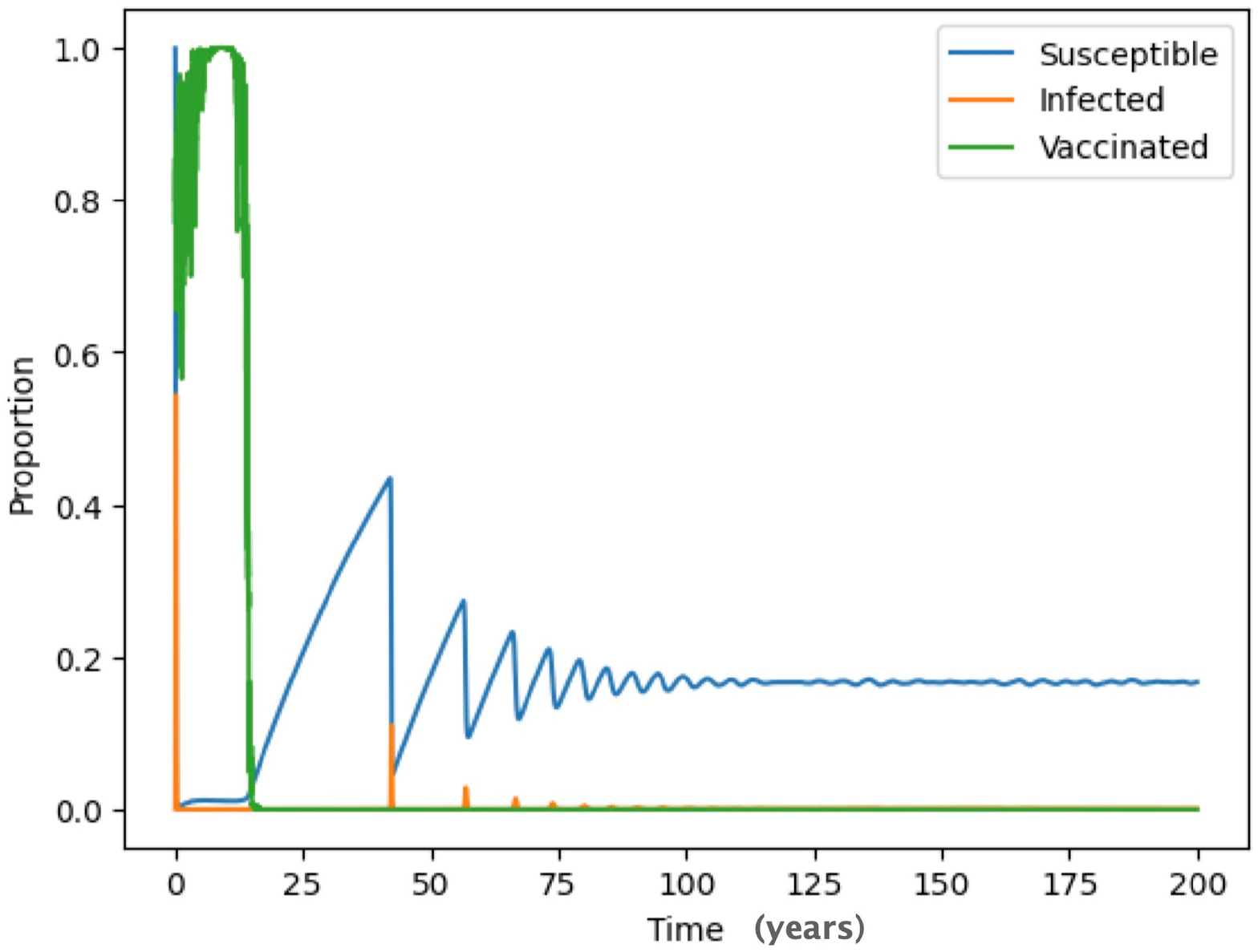}}
     \subfigure[]{\includegraphics[width=8.5 cm]{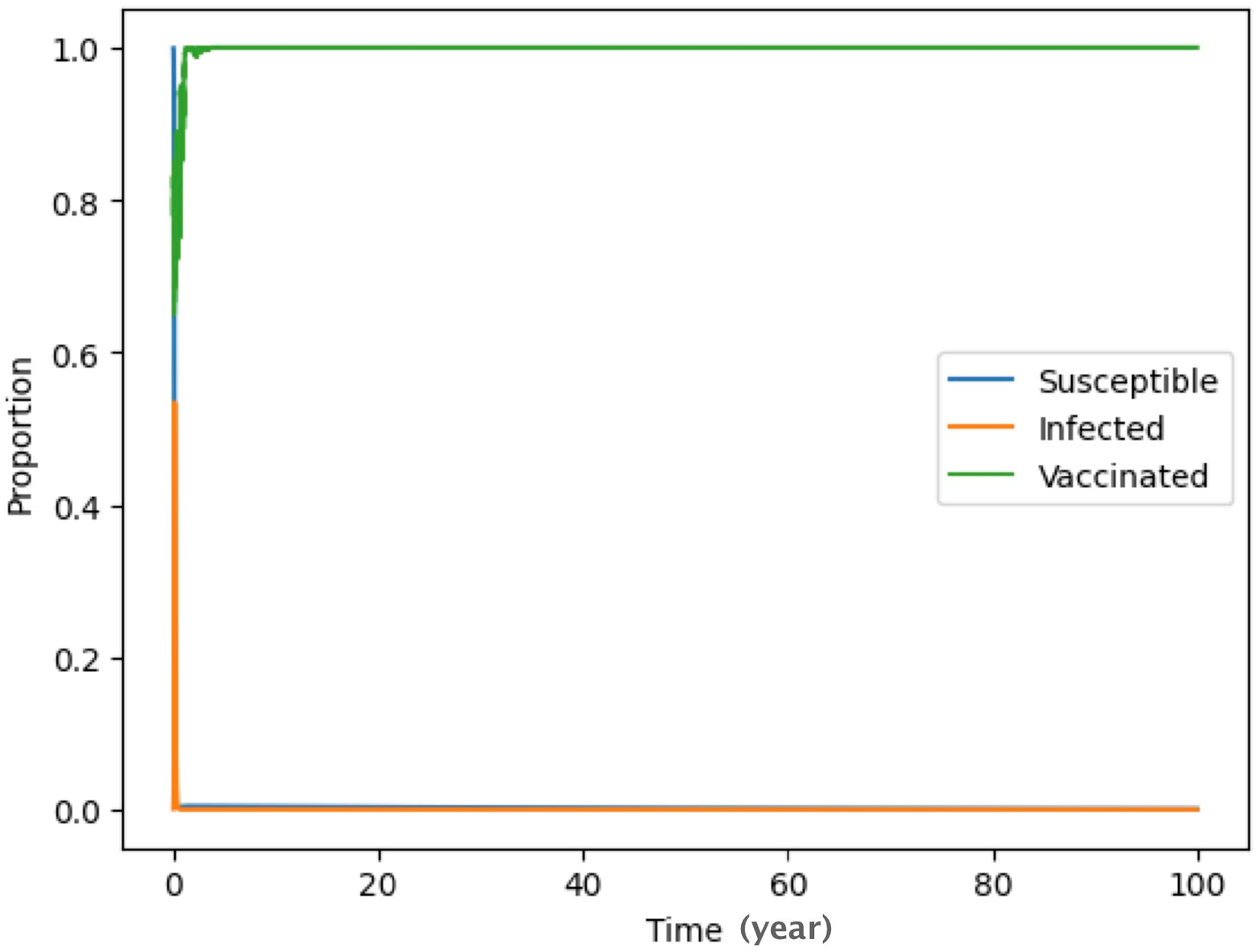}}
    \caption{(a) Parameter plane of $\sigma_2^2$ versus $\sigma_3^2$ based on Theorem \ref{thm:vax}.  Other parameter values are $\beta = 100$,  $\mu = 1/50$, $\gamma = 365/22$, $\kappa = 1.69$,  $\omega = 0.1$, $\delta = 0.5$, and $\sigma_1^2=.16$. In that case, $R_0^s=6$. (b) Simulation of the susceptible, infected, and vaccinated when $\sigma_2^2=1.5$ and $\sigma_3^2=.2$. The endemic equilibrium of $\mathcal{E}_4$ type is stable. (c) Simulation of the susceptible, infected, and vaccinated when $\sigma_2^2=.2$ and $\sigma_3^2=1.5$. The disease-free equilibrium of $\mathcal{E}_1$ type is stable.}
    \label{fig:fig2a}
\end{figure}

Achieving the goal of disease eradication in the vaccination campaign requires minimizing variability in decision-making and promoting consistent, rational choices throughout the population. However, that aim might not be plausible due to external factors. The noise in the perceived utility of the vaccination group could increase further, exacerbating disease eradication. When the social and stochastic external factors that influence utility functions are outside of the control of the policymaker, a policy to change the cost of vaccination is warranted. A stochastic optimal control approach (see e.g. \cite{yong2012stochastic}) could provide a policy for preventing infectious diseases by finding the appropriate cost of vaccination. As such, let the control variable $0<u(t)\leq 1$ be the degree of reduction in the cost of vaccination. Including that function in the stochastic replicator equation gives rise to the following equation
\begin{equation}\label{SIR3.1}
\frac{dx}{dt} = \kappa x (1-x) \left[ -\omega \,(1-u) + I + \delta (2x-1) + \kappa \left( \sigma_3^2 - (\sigma_2^2 + \sigma_3^2) x \right) + \sqrt{\sigma_2^2 + \sigma_3^2} \dot{W}_2(t) \right],
\end{equation}
The details of the stochastic optimal control are given in \ref{optimal}. Figure \ref{fig:fig6} (a) provides a simulation in which $\mathcal{E}_4$ is stable. Full uptake of the vaccine and disease eradication is achieved, see Figure \ref{fig:fig6} (b), with an optimal control $u^*(t)$ that reduces the cost of vaccination, see Figure \ref{fig:fig6} (c). We note that the optimal reduction started to decline smoothly after being a constant over a relatively short period of time.

\begin{figure}[H]
    \centering
     \subfigure[]{\includegraphics[width=8.5 cm]{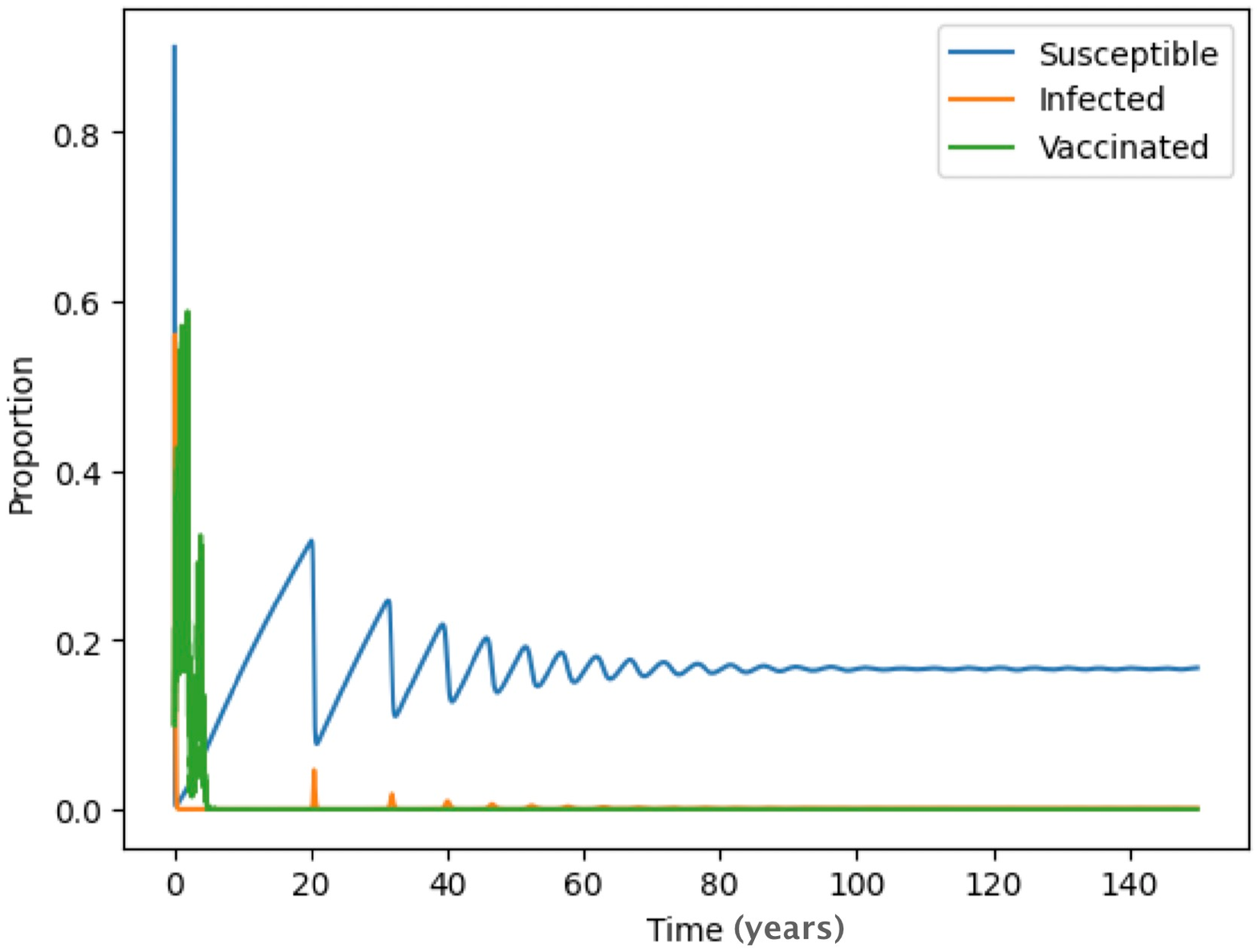}} 
     \subfigure[]{\includegraphics[width=8.5 cm]{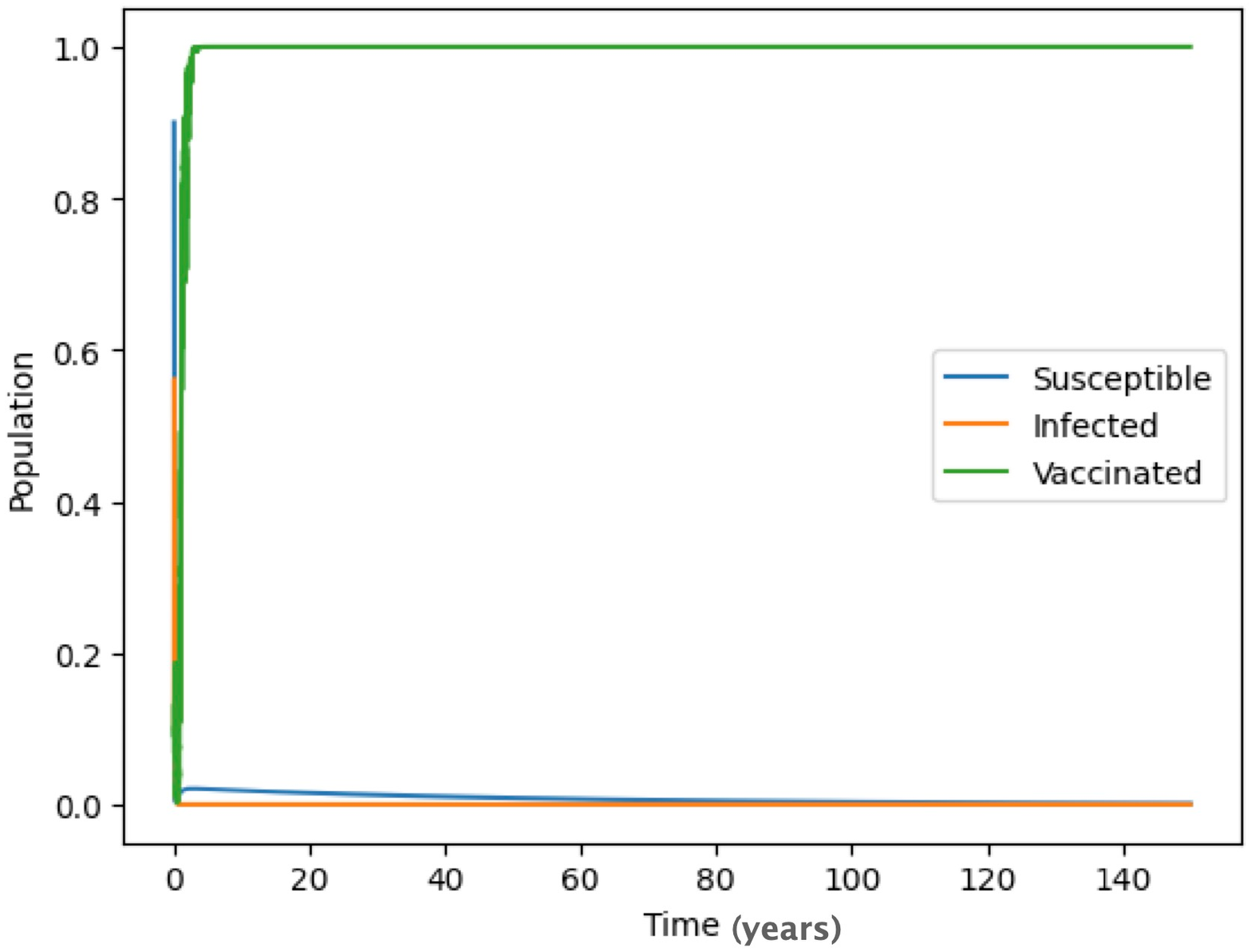}} 
     \subfigure[]{\includegraphics[width=8.5 cm]{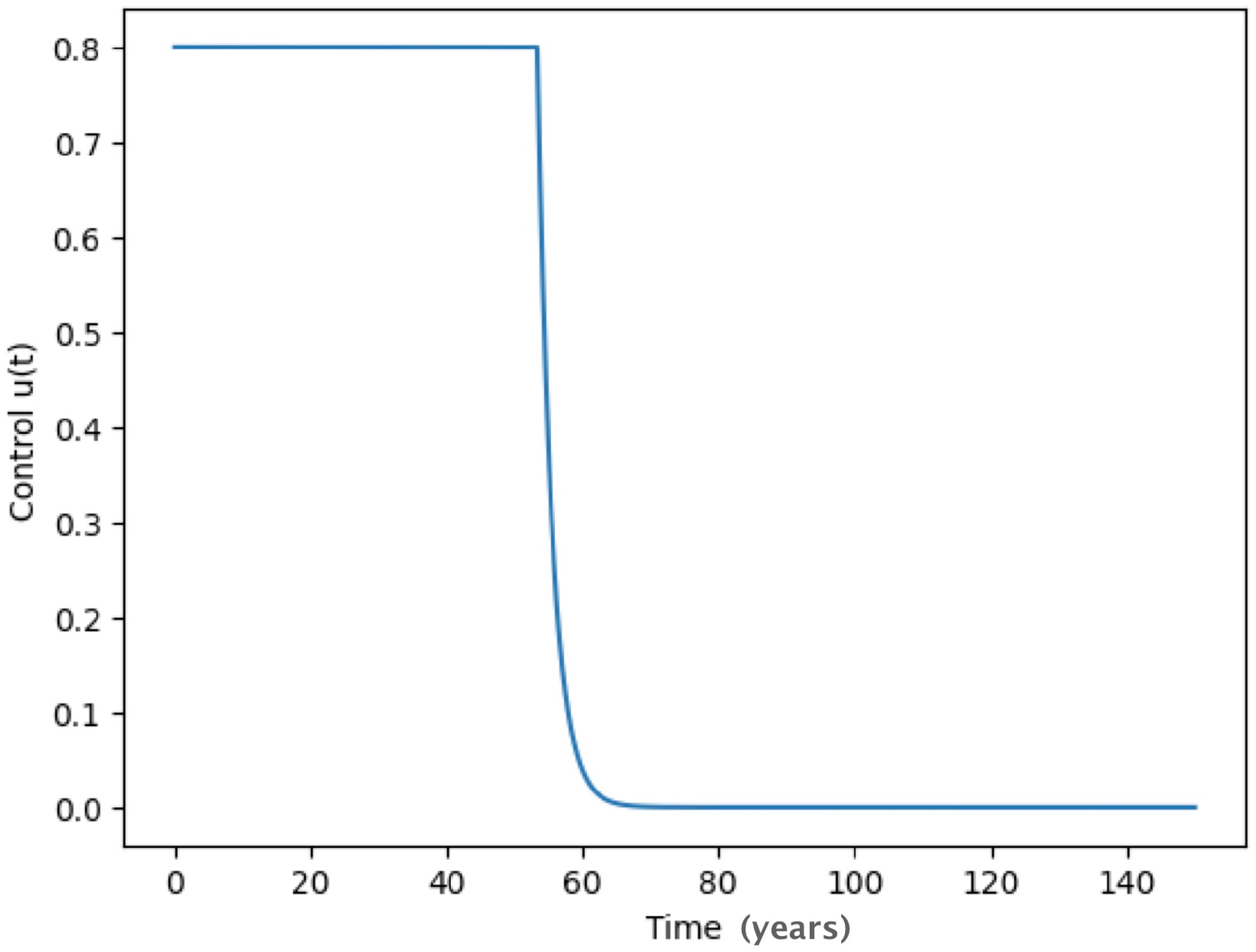}} 
    \caption{Simulation of the susceptible, infected, and vaccinated without control in (a) and with control in (b). (c) The control $u^*(t)$ is shown for $u_{max}=.8$, $T_f=150$, $\alpha_1=0$,  $\alpha_2=1000$, and $\alpha_3=100$. The rest of the parameter values are $\beta = 100$,  $\mu = 1/50$, $\gamma = 365/22$, $\kappa = 1.69$,  $\omega = 2$, $\delta = 0.1$, $\sigma_1^2=.01$, $\sigma_2^2=.5$ and $\sigma_3^2=1.4$. The initial values are $S(0)=.9$, $I(0)=.1$ and $x(0)=.1$. In that case, $R_0^s=6$.}
    \label{fig:fig6}
\end{figure}

It would be difficult to control the disease without that discount in the cost of vaccination, especially with large noise in the vaccination group's utility function. The reason is in two scenarios: the first scenario is when $\omega+\kappa\frac{\sigma_2^2-\sigma_3^2}{2}>I_0-\delta$ and $\omega+\kappa\frac{\sigma_2^2-\sigma_3^2}{2}>\delta$, in which vaccination is completely rejected and the disease becomes endemic; see Figure \ref{fig:fig2a} (b). The second scenario is where $\omega+\kappa\frac{\sigma_2^2-\sigma_3^2}{2}>I_0-\delta$ and $\omega+\kappa\frac{\sigma_2^2-\sigma_3^2}{2}<\delta$, in which case both disease eradication and disease endemicity become stationary; see Figure \ref{fig:fig2b}. In Figure \ref{fig:fig2b} (a) the disease dies, while in Figure \ref{fig:fig2b} (b) the disease persists while all the values of the parameters and initial conditions are the same.

\begin{figure}[H]
    \centering
     \subfigure[]{\includegraphics[width=8.5 cm]{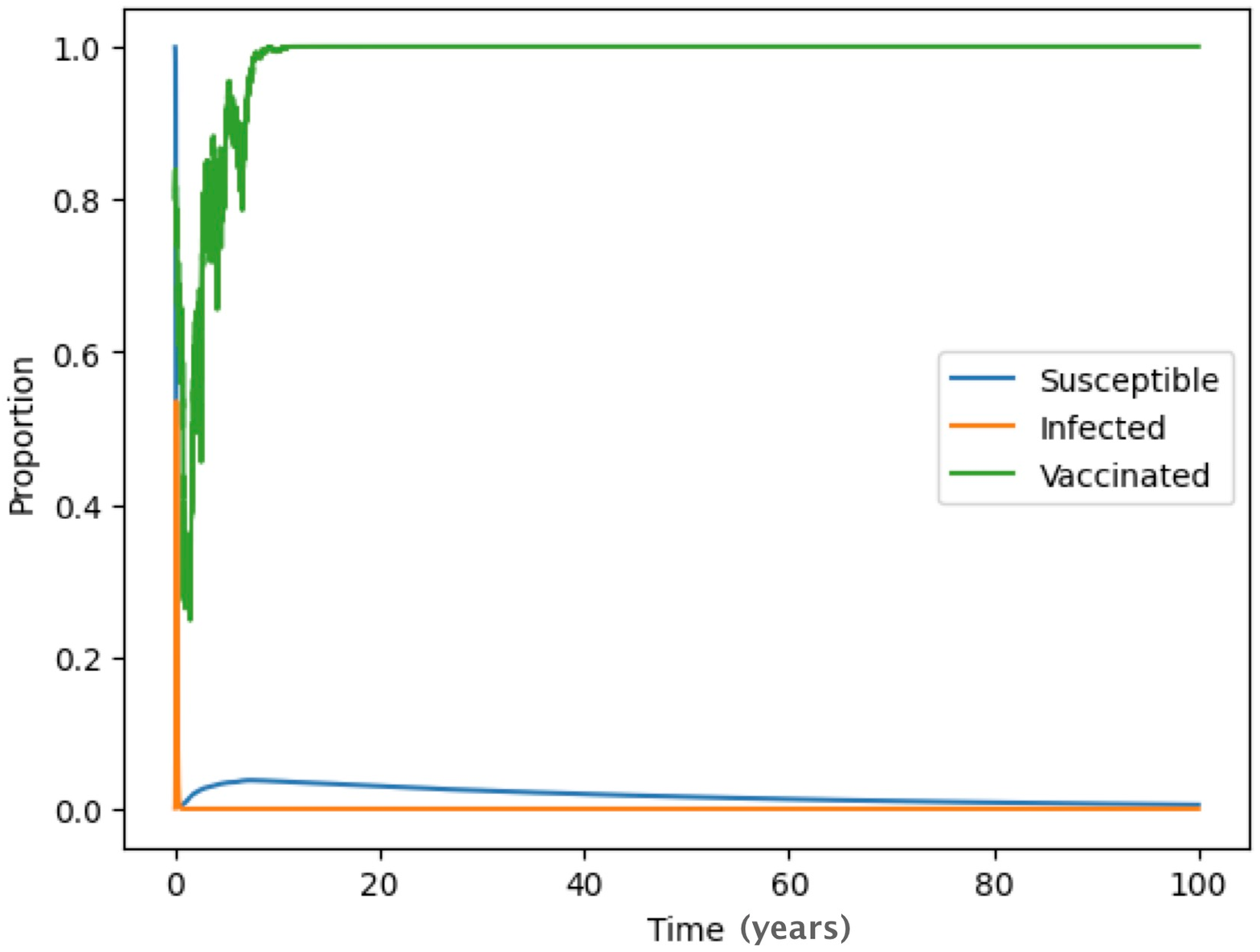}} 
     \subfigure[]{\includegraphics[width=8.5 cm]{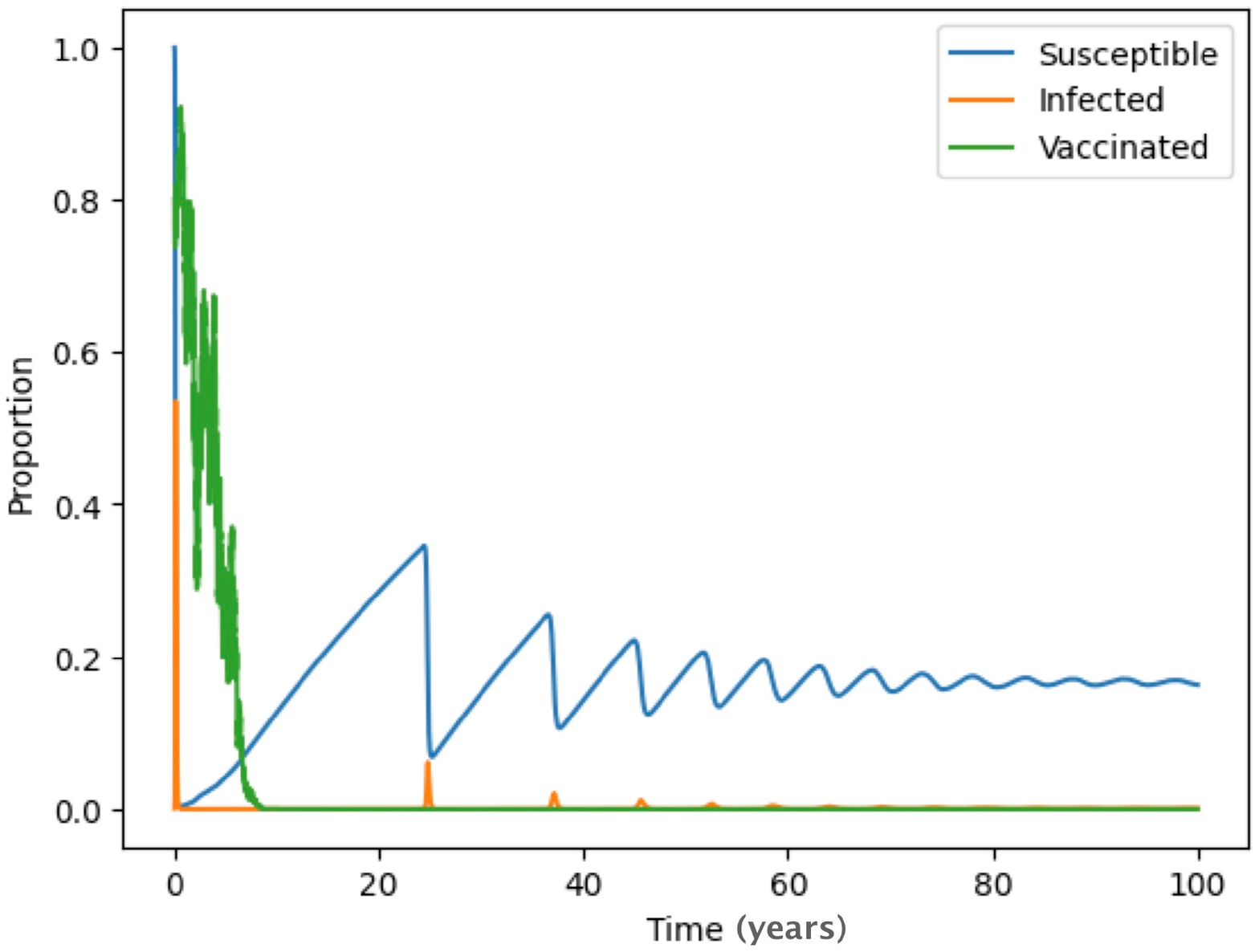}} 
    \caption{Simulation of the susceptible, infected, and vaccinated when $\sigma_2^2=.15$ and $\sigma_3^2=.2$ in (a) and (b) with initial $x(0)=.8$. The disease-free equilibrium of $\mathcal{E}_1$ type and the disease endemic equilibrium of $\mathcal{E}_4$ type are stable -- a Bernoulli stationary distribution. The rest of the parameter values are $\beta = 100$,  $\mu = 1/50$, $\gamma = 365/22$, $\kappa = 1.69$,  $\omega = 0.1$, $\delta = 0.5$, and $\sigma_1^2=.16$. In that case, $R_0^s=6$.}
    \label{fig:fig2b}
\end{figure}

As shown in Figure \ref{fig:fig2b}, when both parts of Theorem \ref{thm:vax} are true, namely $\omega+\kappa\frac{\sigma_2^2-\sigma_3^2}{2}>I_0-\delta$ and $\omega+\kappa\frac{\sigma_2^2-\sigma_3^2}{2}<\delta$, then the bistability of both equilibria $x=0$ and $x=1$ (impermanence) occurs. A further examination of the stability of both equilibria was performed using a simulation study. Figure \ref{fig:fig2} shows the proportion of times $x(t)$ converges to $0$ and to $1$ in $200$ stochastic simulations of the model at different values of $\sigma_2^2$ and $\sigma_3^2$ with initial values $x(0)=0.1,0.2,\ldots,0.9$. The bistability region shows a Bernoulli stationary distribution whose probabilities depend on the noises in the perceived utilities $\sigma_2^2$ and $\sigma_3^2$, as well as the initial acceptance of vaccination $x(0)$. The size of the bistability region could be seen to depend on $x(0)$ as shown through the panels of Figure \ref{fig:fig2}. We notice that as initial acceptance of the vaccine $x(0)$ increases, the region in which full vaccination is stable in the $\sigma_2^2-\sigma_3^2$ plane expands. In other words, noises are more detrimental to vaccination campaigns if initial vaccine acceptance is low.
\newpage
\begin{figure}[H]
\begin{center}

\includegraphics[width=0.5\textwidth]{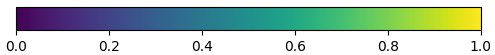} 
\begin{tabular}{|c|c|c|} \hline 
$x(0)=0.1$    & $x(0)=0.2$  & $x(0)=0.3$ \\
\includegraphics[width=0.3\textwidth]{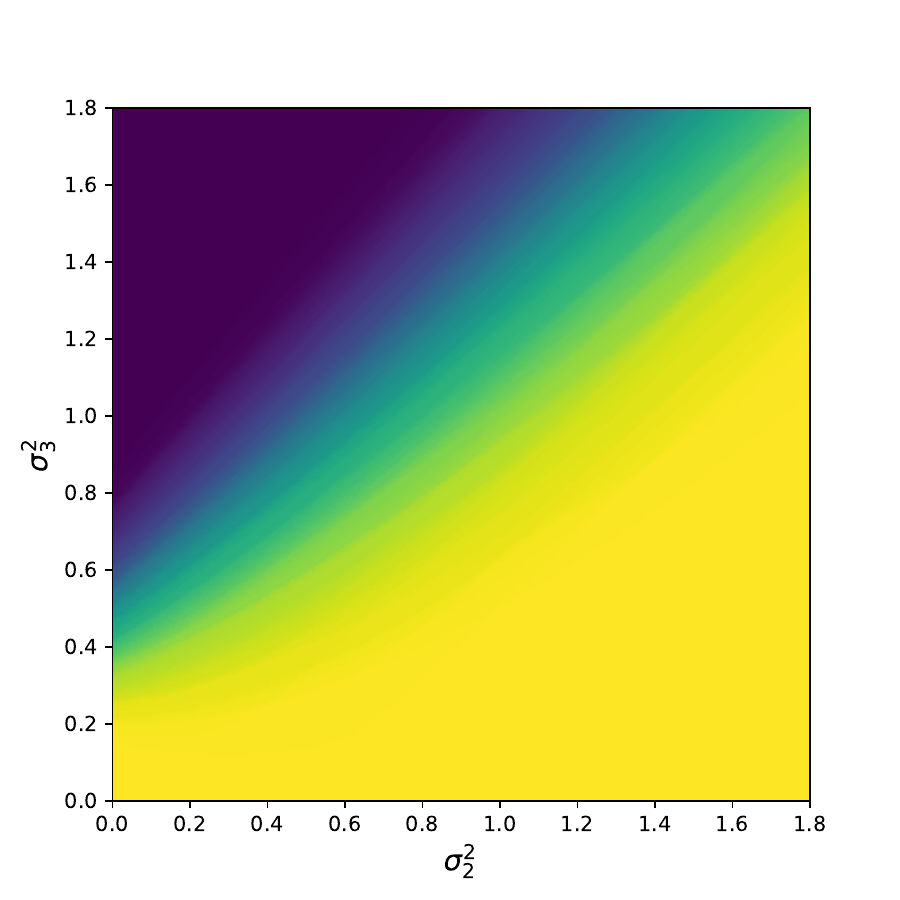}     & \includegraphics[width=0.3\textwidth]{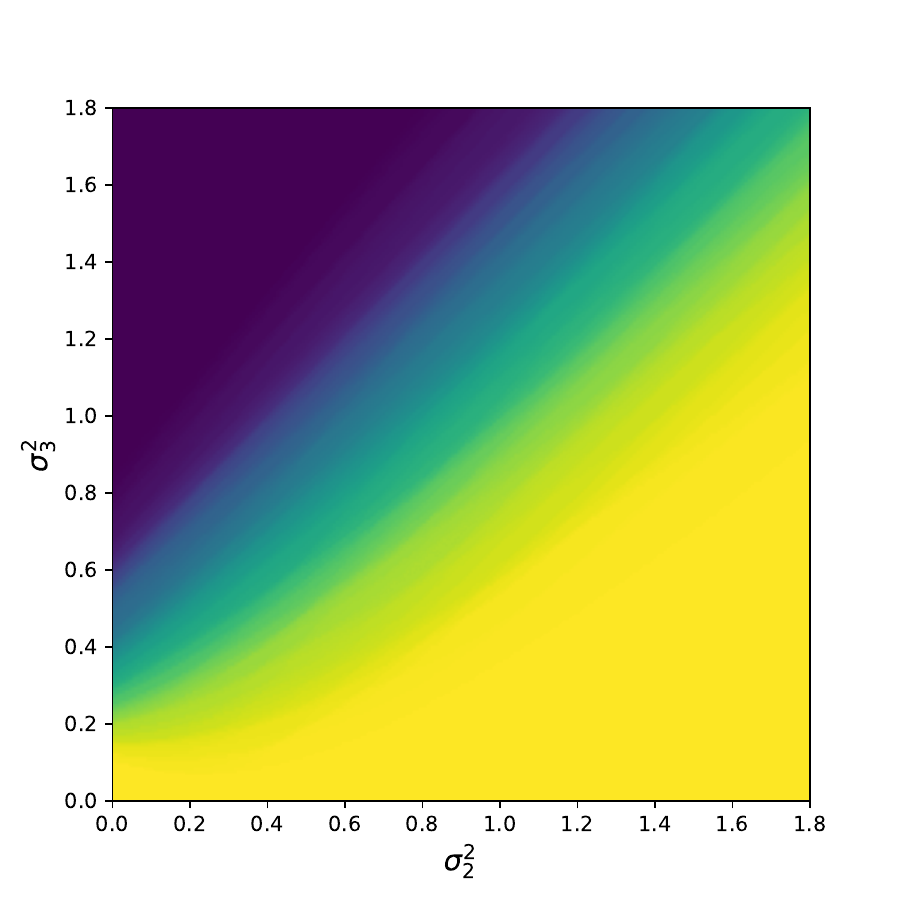}  & \includegraphics[width=0.3\textwidth]{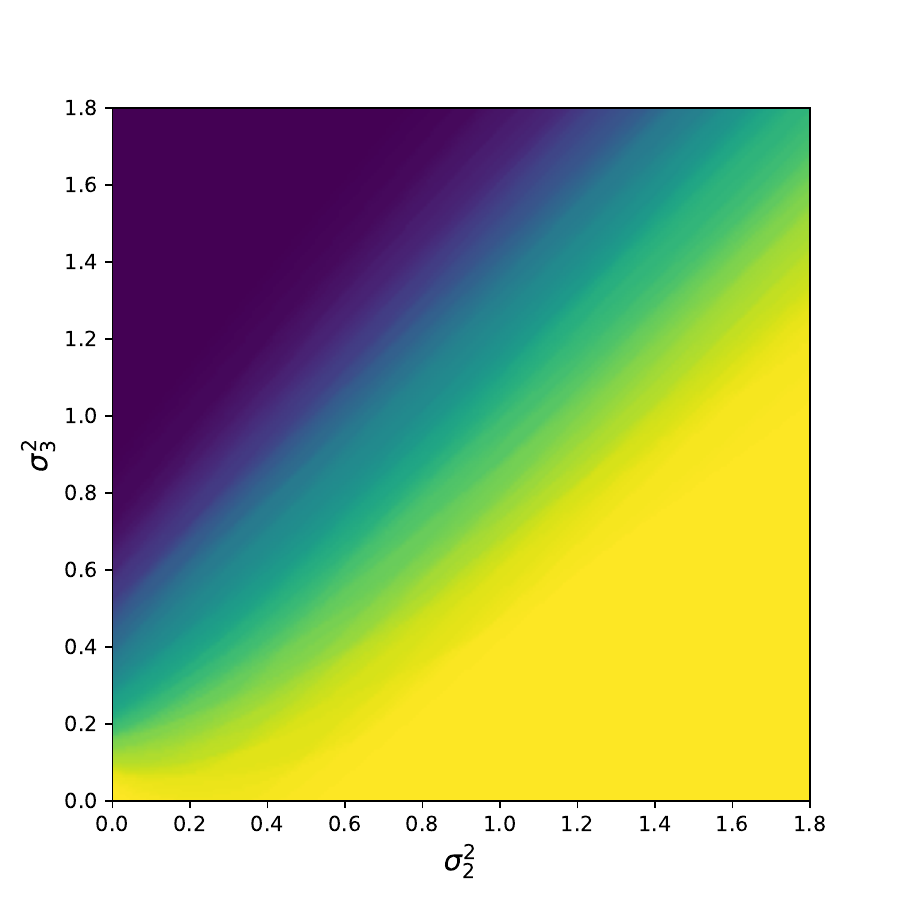} \\[2ex]
\hline
$x(0)=0.4$    &  $x(0)=0.5$  &  $x(0)=0.6$ \\
\includegraphics[width=0.3\textwidth]{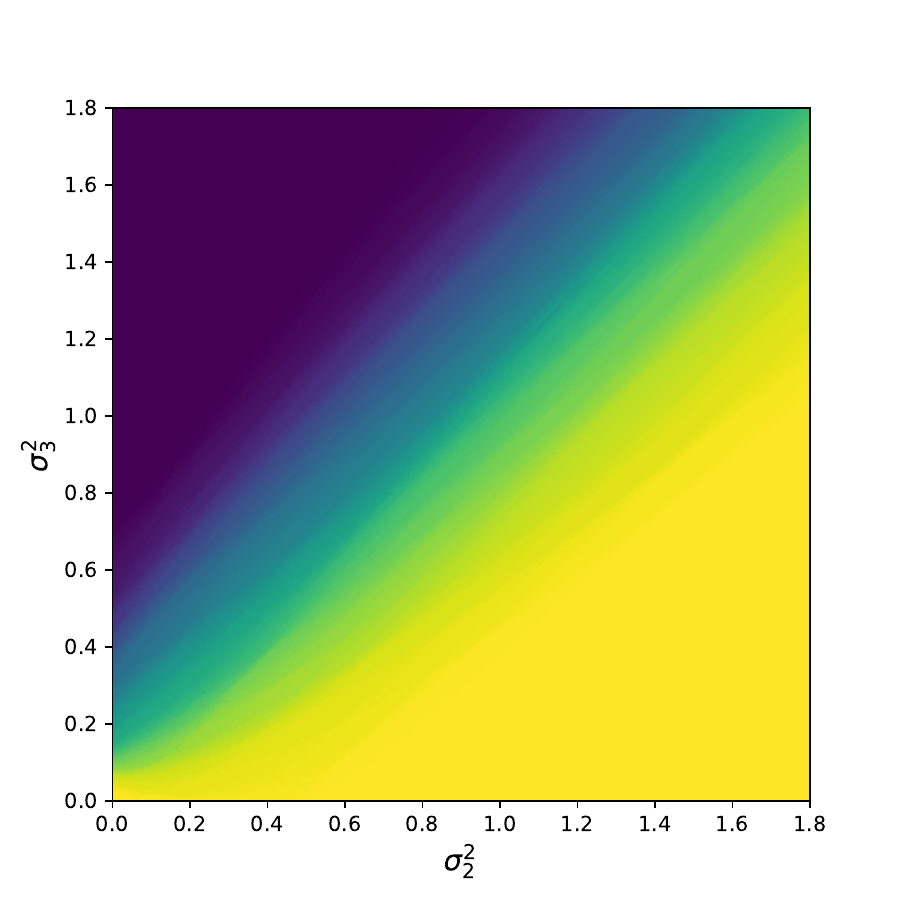}     & \includegraphics[width=0.3\textwidth]{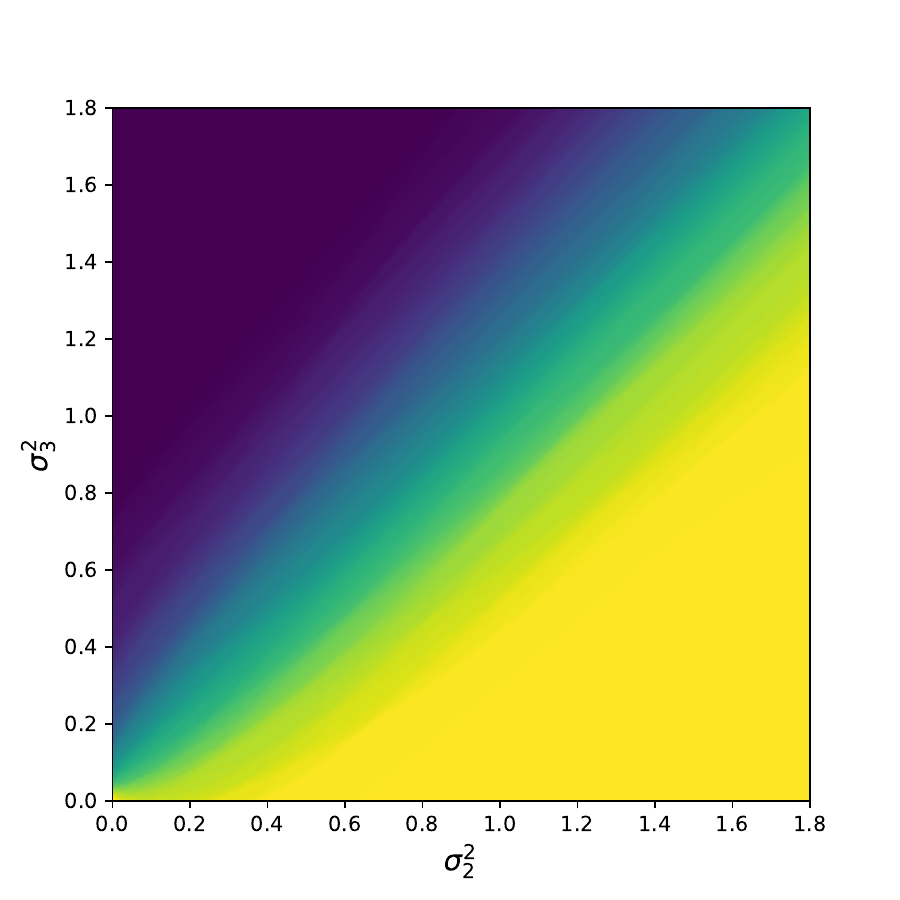}  & \includegraphics[width=0.3\textwidth]{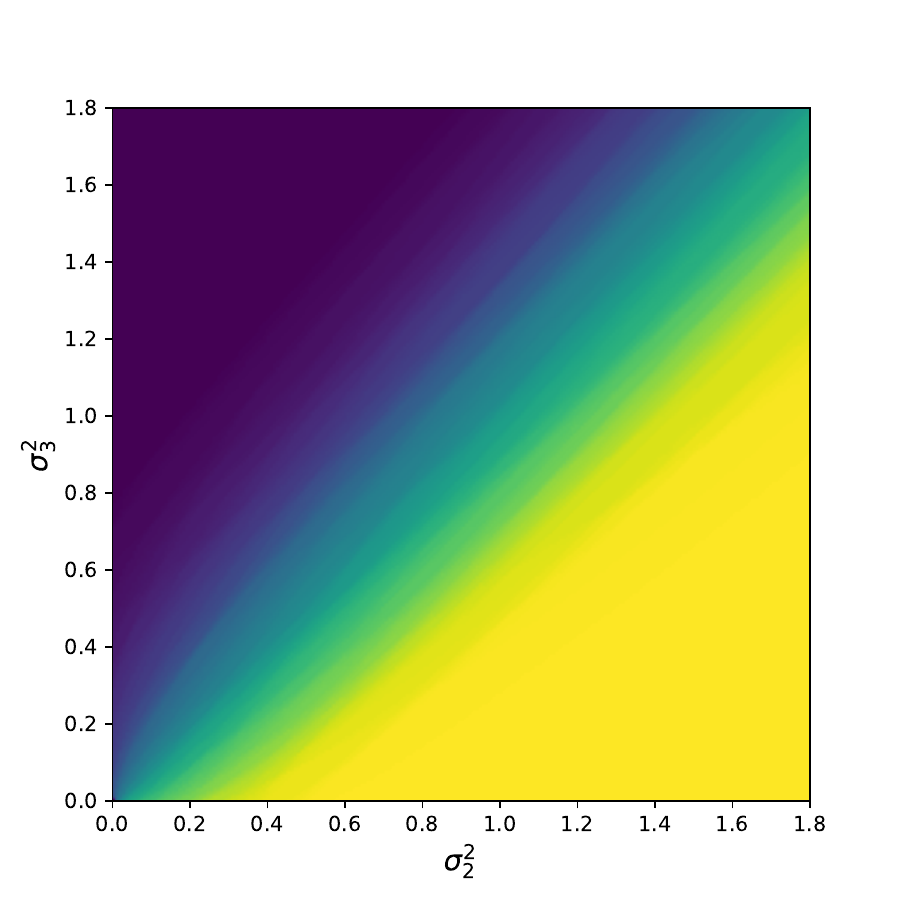} \\[2ex]
\hline
$x(0)=0.7$    & $x(0)=0.8$  &  $x(0)=0.9$ \\
\includegraphics[width=0.3\textwidth]{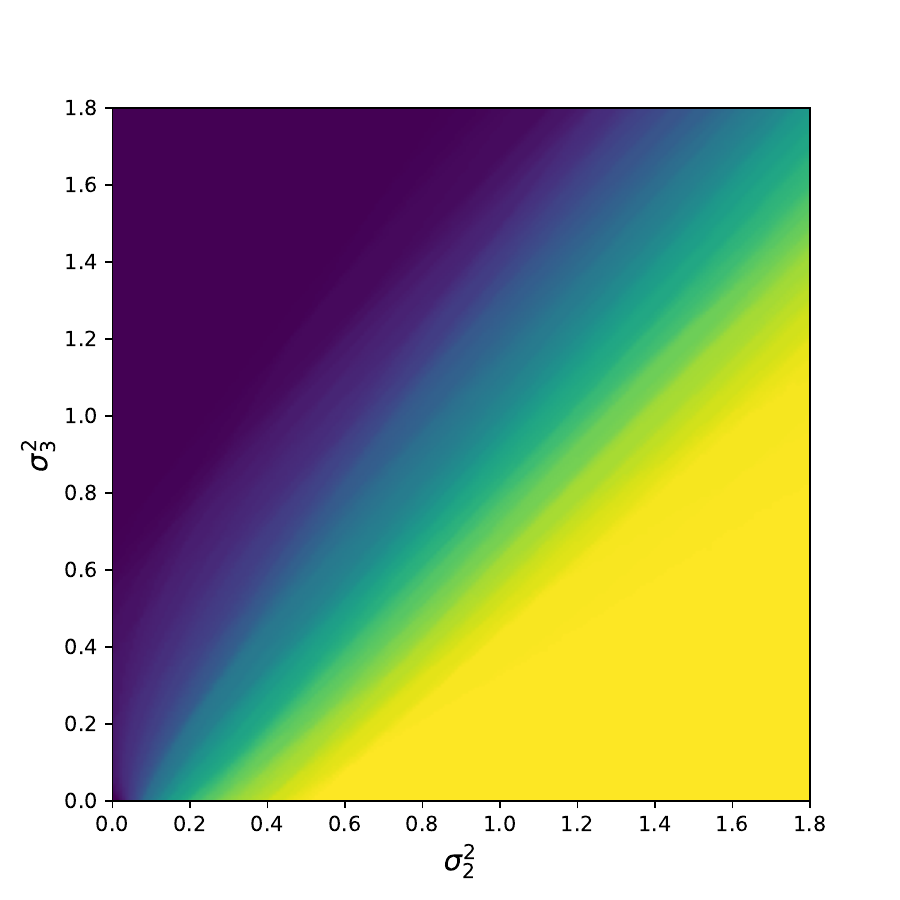}     & \includegraphics[width=0.3\textwidth]{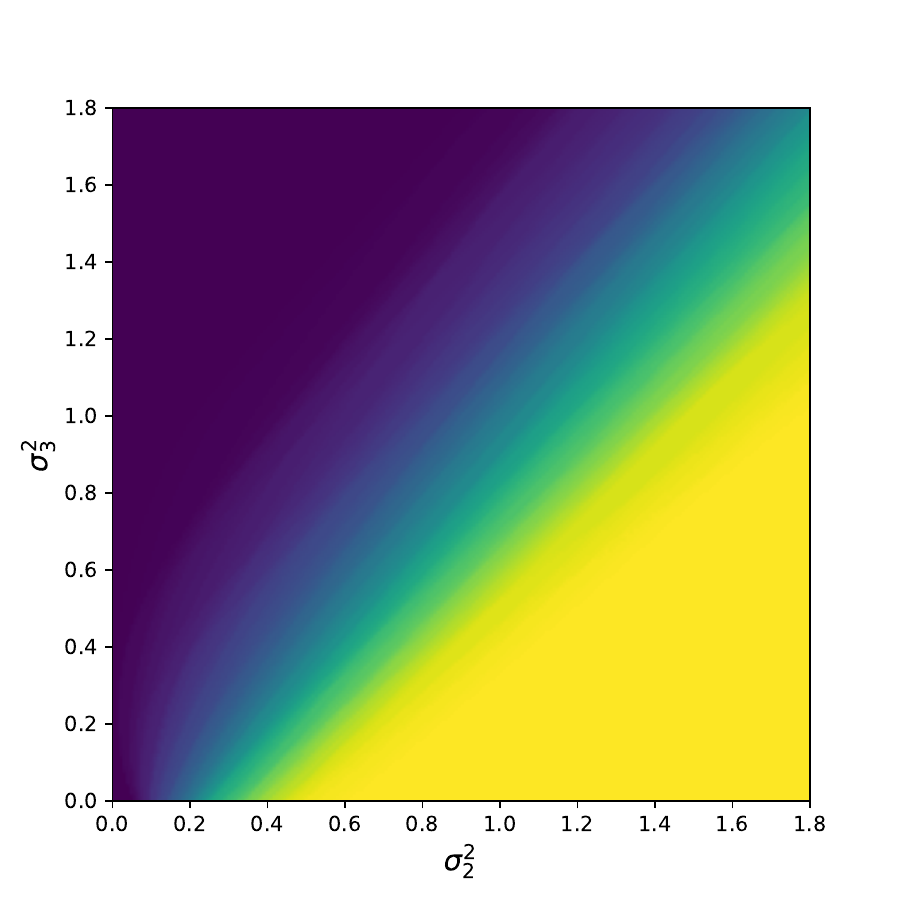}  & \includegraphics[width=0.3\textwidth]{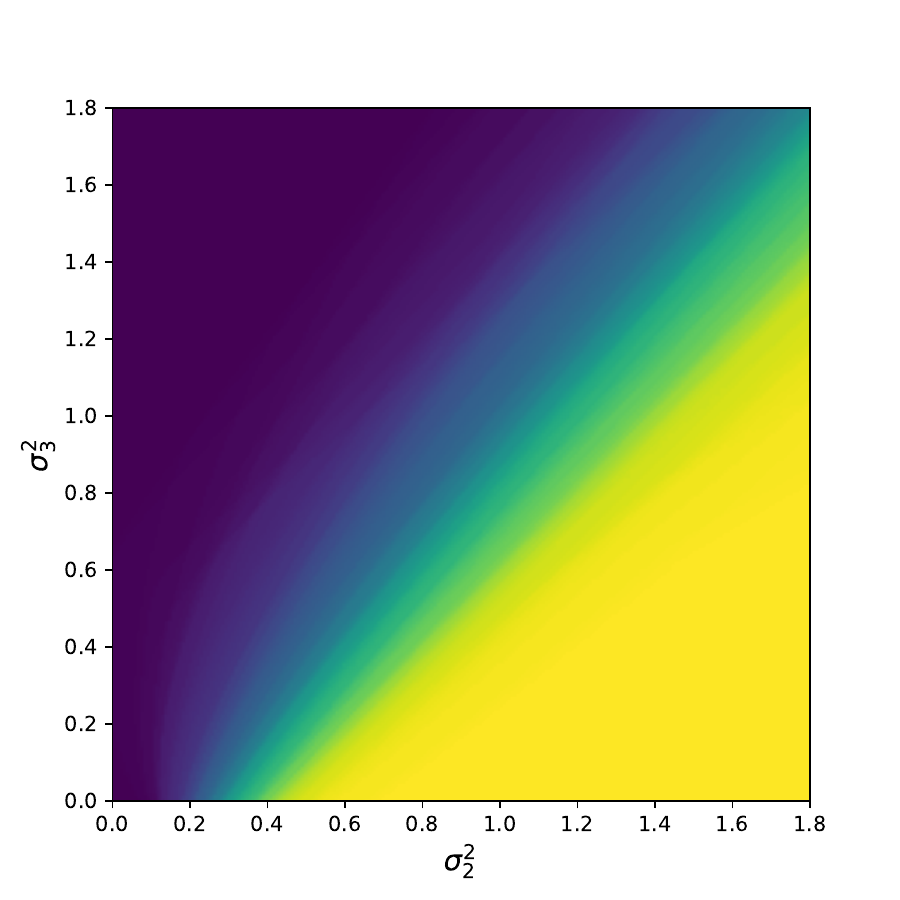} \\
\hline

\end{tabular}
    \caption{ Estimate of the probability $P\left(\lim_{t\to\infty}x(t)=0\vert x(0)\right)$ using $200$ simulation runs at $T=100$, $dt = 0.001$ and different initial values of $x(0)$, with parameter values $\beta = 100$,  $\mu = 1/50$, $\gamma = 365/22$, $\kappa = 1.69$,  $\omega = 0.1$, $\delta = 0.5$, $\sigma_1^2 = .16$, and initial values $S(0)=0.4$ and $I(0)=0.4$.}\label{fig:fig2}
\end{center}
\end{figure}

The following two theorems are generalizations to the disease dynamics in \cite{oraby1}.
The following theorem identifies the boundaries of the limiting temporal means of processes when $R_0^s>1$. It also shows the case when the disease cannot be eradicated at a low vaccination level; namely when its temporal mean of vaccine acceptance/uptake is below $1-\frac{1}{R_0^s}$.

\begin{theorem}\label{thm23}
If $R_0^s>1$ and $\bar{x}(t)\to x_0$, then a disease-endemic equilibrium will persist in the temporal mean. That is
\begin{enumerate}
\item For all $x_0\in [0,1]$, 
\begin{equation} \label{endemic1} 
 \dfrac{\mu\,(1-x_0)}{\mu+\beta} \leq \liminf_{t\to \infty} \bar{S}(t)\leq \limsup_{t\to \infty} \bar{S}(t)\leq 1-x_0
\end{equation} and
    \item If $x_0<1-\frac{1}{R_0^s}$, then \begin{equation}\label{endemic2}
    \liminf_{t \to \infty} \bar{I}(t)\geq \frac{\mu}{\mu+\gamma}[1-\frac{1}{R_0^s}-x_0]\end{equation} 
    and 
    \begin{equation}\label{upS}
    \limsup_{t \to \infty} \bar{S}(t)\leq \frac{1}{R_0^s}.
\end{equation}
Equation \eqref{upS} holds true also when $x_0\geq 1-\frac{1}{R_0^s}$ from the upper inequality in Part (1).
    \item If $x_0<\frac{\beta}{\beta-\sigma_1^2}(1-\frac{1}{R_0^s})$ and $\beta>\sigma_1^2$ then \begin{equation}\label{endemic3}
    \limsup_{t\to \infty} \bar{I}(t)\leq \frac{\mu}{\mu+\gamma}[\frac{\beta}{\beta-\sigma_1^2}(1-\frac{1}{R_0^s})-x_0].
    \end{equation} 
\end{enumerate}
\end{theorem}

\begin{proof}
Since
\begin{equation*} 
\frac{d (S+I)(t)}{dt}=\mu(1-x(t))-\mu(S+I)(t)-\gamma I(t) 
\end{equation*}
then
\begin{equation}\label{bareq}
   \bar{S}(t)=1-\phi(t)-\bar{x}(t)-(1+\frac{\gamma}{\mu}) \bar{I}(t)
\end{equation}
where $\phi(t)=\frac{1}{\mu t}\left(S(t)+I(t)-S(0)-I(0) \right)$. Notice that $\lim_{t\to \infty} \phi(t)=0$ $a.s.$ Since, $\bar{S}(t)\leq 1-\phi(t)-\bar{x}(t)$ for all $t>0$. Therefore, 
$$\limsup_{t\to \infty} \bar{S}(t)\leq 1-x_0.$$

From equation \eqref{SIR} 
$$ 
\frac{dS}{dt} = \mu\,(1-x)- \beta \, S\,I-\mu \, S-\sigma_1 S\,I\,\dot{W}_1   \geq \mu\,(1-x)- (\beta +\mu) \, S-\sigma_1 S\,I\,\dot{W}_1.$$
Then,
$$ 
\frac{S(t)-S(0)}{t}  \geq \mu\,(1-\bar{x}(t))- (\beta +\mu) \, \bar{S}(t)-\frac{1}{t} M_3(t),$$
where $M_3(t)=\int_0^t \sigma_1 S(u)\,I(u)\,d{W}_1(u)$. Notice that $M_3(t)$ is a local continuous martingale with $M_3(0)=0$ and $E(M_3(t)^2)=\sigma_1^2 \int_0^t (S(u) I(u))^2 du\leq \sigma_1^2 t$. Since $\limsup_{t\to \infty} \dfrac{E(M_3(t)^2)}{t}\leq \sigma_1^2<\infty$, then by the Strong Law of Large Numbers, $\limsup_{t\to \infty} \dfrac{M_3(t)}{t}=0$ almost surely, see \cite{mao2007stochastic} and Lemma \ref{martin} in the Appendix.
But since $x_0<1$, and $S(t)-S(0)\leq 1$ then \eqref{endemic1} follows.

To prove part (2), note that
\begin{eqnarray*}
\dfrac{d\log(I(t))}{dt}&=& \beta S(t)-(\mu+\gamma +\frac12 \sigma_1^2) +\frac12 \sigma_1^2(1-S(t)^2) +\sigma_1 S\,\dot{W}_1\\
&\geq& \beta S(t)-(\mu+\gamma +\frac12 \sigma_1^2) +\sigma_1 S\,\dot{W}_1
\end{eqnarray*}
Thus
\begin{eqnarray*}
\dfrac{\log(I(t))-\log(I(0))}{t}
&\geq& \beta [1-\phi(t)-\bar{x}(t)-(1+\frac{\gamma}{\mu}) \bar{I}(t)]-(\mu+\gamma +\frac12 \sigma_1^2) +\frac{1}{t} M(t)
\end{eqnarray*}
and so
\begin{eqnarray*}
\dfrac{\log(I(t))}{t}
&\geq& \dfrac{\log(I(0))}{t}+\beta [1-\phi(t)-\bar{x}(t)-(1+\frac{\gamma}{\mu}) \bar{I}(t)]-(\mu+\gamma +\frac12 \sigma_1^2) +\frac{1}{t} M(t).
\end{eqnarray*}
Thus by Lemma \ref{Ap1} in the Appendix,
\begin{equation*}
    \liminf_{t \to \infty} \bar{I}(t)\geq \frac{\mu}{\mu+\gamma}[1-\frac{1}{R_0^s}-x_0]
\end{equation*}
if $x_0<1-\frac{1}{R_0^s}$. In which case, since
\begin{equation*}
   \bar{S}(t)=1-\phi(t)-\bar{x}(t)-(1+\frac{\gamma}{\mu}) \bar{I}(t).
\end{equation*}
Then
\begin{equation*}
    \limsup_{t \to \infty} \bar{S}(t)\leq \frac{1}{R_0^s}
\end{equation*}
Notice that $\frac{1}{R_0^s}>\frac{1}{R_0}$ so it is more than its deterministic counterpart. 

To prove part (3), similarly, if $x_0<\frac{\beta}{\beta-\sigma_1^2}(1-\frac{1}{R_0^s})$ and $\beta>\sigma_1^2$ then 
    \begin{equation*}\label{endemic3S}
    \liminf_{t\to \infty} \bar{S}(t)\geq 1-\frac{\beta}{\beta-\sigma_1^2}(1-\frac{1}{R_0^s})
    \end{equation*} 

Again,
\begin{equation*}
d\log(I(t))= (\beta \, S-(\mu+\gamma) -\frac12 \sigma_1^2 S^2) dt+\sigma_1 S  dW_1(t)
\end{equation*}
and by equation \eqref{bareq}
\begin{equation*}
\frac{1}{t}\log(I(t))-\frac{1}{t}\log(I(0))= \beta \, \bar{S}(t)-(\mu+\gamma) -\frac12 \sigma_1^2 \overline{S^2}(t) +\frac{1}{t}M(t)
\end{equation*}
where $M(t)=\sigma_1 \int_0^t S(u)  dW_1(u)$ is a local continuous martingale, with $M(0)=0$, see \cite{mao2007stochastic} and Lemma \ref{martin} in the Appendix. Since, $(\bar{S}(t))^2\leq \overline{S^2}(t)$, then
\begin{align*}
\frac{1}{t}\log(I(t))&\leq \frac{1}{t}\log(I(0))+\beta \, [1-\phi(t)-\bar{x}(t)-(1+\frac{\gamma}{\mu}) \bar{I}(t)]-(\mu+\gamma) \\ &-\frac12 \sigma_1^2 [1-\phi(t)-\bar{x}(t)-(1+\frac{\gamma}{\mu}) \bar{I}(t)]^2 +\frac{1}{t}M(t) \\ &\leq \Psi(t)+\beta \, [1-\bar{x}(t)]-(\mu+\gamma)-\beta \,(1+\frac{\gamma}{\mu}) \bar{I}(t)-\frac12 \sigma_1^2 (1-\phi(t))^2 \\
& +\sigma_1^2 (\bar{x}(t)+(1+\frac{\gamma}{\mu}) \bar{I}(t))-\frac12 \sigma_1^2 (\bar{x}(t)+(1+\frac{\gamma}{\mu}) \bar{I}(t))^2 \\ &\leq \Xi(t)-(\beta-\sigma_1^2) \,(1+\frac{\gamma}{\mu}) \bar{I}(t)
\end{align*}
where
\begin{equation*}
    \Xi(t)=\Psi(t)+\beta \, [1-\bar{x}(t)]-(\mu+\gamma)-\frac12 \sigma_1^2 (1-\phi(t))^2+\sigma_1^2 \bar{x}(t)
\end{equation*}
and
\begin{align*}
    \Psi(t)&=\frac{1}{t}\log(I(0))-\beta\phi(t)-\sigma_1^2 \phi(t)(\bar{x}(t)+(1+\frac{\gamma}{\mu}) \bar{I}(t))
    +\frac{1}{t}M(t).
\end{align*}
If we show that $\lim_{t\to \infty}\Xi(t)=\beta(1-\frac{1}{R_0^s})-(\beta-\sigma_1^2)x_0$ $a.s.$,
then Lemma \ref{Ap1} in the Appendix will imply that
$$\limsup_{t\to \infty} \bar{I}(t)\leq \frac{\mu}{\mu+\gamma}[\frac{\beta}{\beta-\sigma_1^2}(1-\frac{1}{R_0^s})-x_0]$$ 
if $x_0<\frac{\beta}{\beta-\sigma_1^2}(1-\frac{1}{R_0^s})$. But $\lim_{t\to \infty}\frac{1}{t}\log(I(0))=0$ $a.s.$ and $\lim_{t\to \infty}\frac{1}{t}M(t)=0$ $a.s.$ Finally,
$$\vert -\beta\phi(t)-\sigma_1^2 \phi(t)(\bar{x}(t)+(1+\frac{\gamma}{\mu}) \bar{I}(t)) \vert\leq \beta\vert \phi(t)\vert+\sigma_1^2 \vert\phi(t) \vert (1+(1+\frac{\gamma}{\mu}) )$$ and $\lim_{t\to \infty}\phi(t)=0$ $a.s.$ complete this part.
\end{proof}

In \cite{oraby1}, where the deterministic version of the model is studied, the partial vaccination and persistence of the disease equilibrium $
\displaystyle \mathcal{E}_5' \equiv \left(\frac{1}{R_0},\frac{\mu}{\mu+\gamma}\left(1-\frac{1}{R_0}-x_5'\right),x_5'\right)$, where $$x_5'=\dfrac{\mu\left(1-\dfrac{1}{R_0}\right)-\left(\delta+\omega\right)(\mu+\gamma)}{\mu-2\delta(\mu+\gamma)} $$
is locally asymptotically stable when 
\begin{equation}\label{new_ineq}
    \delta(1-\frac{2}{R_0}) <\omega<-\delta+\frac{\mu}{\mu+\gamma}\,(1-\frac{1}{R_0})
\end{equation}
The left hand side of inequality \eqref{new_ineq} is equivalent to $x_5'<1-\frac{1}{R_0}$, the existence condition of $\mathcal{E}_5'$. Part 2 of Theorem \ref{thm23}, exhibits a similar result in which if the temporal mean of $x(t)$ converges to $x_0<1-\frac{1}{R_0^s}$, then the disease persists as in the deterministic model even when $x_0>0$. In which case, parts 2 and 3 of Theorem \ref{thm23} give \begin{equation}\label{endemic4}
   0<\frac{\mu}{\mu+\gamma}[1-\frac{1}{R_0^s}-x_0] \leq \liminf_{t \to \infty} \bar{I}(t)  \leq \limsup_{t \to \infty} \bar{I}(t) \leq \frac{\mu}{\mu+\gamma}[\frac{\beta}{\beta-\sigma_1^2}(1-\frac{1}{R_0^s})-x_0]. 
   \end{equation} 
Moreover, if $\sigma_1^2=0$ while $R_0>1$, then \begin{equation}\label{endemic5}
    \lim_{t \to \infty} \bar{I}(t) =\frac{\mu}{\mu+\gamma}[1-\frac{1}{R_0}-x_0]. 
   \end{equation}    
When $x_0=0$, the disease persists in a way similar to $\mathcal{E}_4$.

When $R_0^s>1$ some non-deterministic limits of the system occur under the condition that $\dfrac{\sigma_1^2}{\beta}<\dfrac{R_0}{2}$; see, e.g. \cite{Lahrouz2014}. Even when $\sigma_1^2=0$ and $R_0>1$, in view of equation \eqref{endemic5}, disease prevalence still fluctuates; see Figure \ref{fig:fig4} (a). The following theorem provides boundaries for those stochastic fluctuations. 

\begin{theorem}\label{thm24}
 If $R_0^s>1$, and $\dfrac{\sigma_1^2}{\beta}<\dfrac{R_0}{2}$, then
\begin{enumerate}
    \item $ \liminf_{t\to \infty} S(t)\leq s_d \leq \limsup_{t\to \infty} S(t)\leq 1-x_*$ a.s.
    \item $\liminf_{t\to \infty} I(t)\leq (1-s_d-x_*)\frac{\mu}{\mu+\gamma} $ a.s.
    \item $(1-s_d-x^*)\frac{\mu}{\mu+\gamma} \leq \limsup_{t\to \infty} I(t) \leq (1-s_d)$ a.s.
\end{enumerate}
where $s_d:=\frac{1}{R_0}\,\frac{2}{1+\sqrt{1-\frac{2\sigma_1^2}{\beta R_0}}}\geq \frac{1}{R_0}$. Note that the first part implies that $x_*\leq 1- s_d$.
\end{theorem}
\begin{proof}
Since,
\begin{eqnarray*}
d\log(I(t))&=& (-\frac12 \sigma_1^2 S^2(t)+\beta S(t)-(\mu+\gamma))dt + \sigma_1 S dW_1(t).
\end{eqnarray*}
Let $\phi(z)=-\frac12 \sigma_1^2 z^2+\beta z-(\mu+\gamma)$, which has one zero $s_d$ between zero and one if and only if $\frac{\sigma_1^2}{\beta}<\dfrac{R_0}{2}$ and $R_0^s>1$. The function $\phi(z)$ can have a peak at $z^*=\frac{\beta}{\sigma_1^2}$ such that $s_d<z^* \leq 1$ if and only if $\beta\leq \sigma_1^2$. Since $\phi(z)$ is increasing for $z<z^*$ then there exists $\epsilon>0$ such that $\phi(z)$ is increasing for $z<s_d+\epsilon$. Thus, 
$\phi(z)<\phi(s_d-\epsilon)<0$ for $z<s_d-\epsilon$. 

Assume that $\limsup_{t\to \infty} S(t)<s_d-\epsilon$ on a set $\Omega_1$ with positive probability, and so $0\leq S(u) < s_d-\epsilon<1$ for all $u>t_0(\omega)$ for each $\omega \in \Omega_1$. Thus, $\phi(S(u))<\phi(s_d-\epsilon)<0$ for all $u>t_0(\omega)$ for each $\omega \in \Omega_1$. Then, $\limsup_{t\to \infty} \frac{\log(I(t))}{t}\leq \phi(s_d-\epsilon)<0$ with positive probability. That means $\lim_{t\to \infty} \bar{I}(t)=0$ with positive probability, but that would contradict Theorem \ref{thm23} part (2). 

Now, assume that $\liminf_{t\to \infty} S(t)>s_d+\epsilon$ on a set $\Omega_2$ with positive probability, and so $s_d+\epsilon\leq S(u)  \leq 1$ for all $u>t_0(\omega)$ for each $\omega \in \Omega_2$. Thus, $\phi(S(u))>\phi(s_d+\epsilon)>0$ for all $u>t_0(\omega)$ for each $\omega \in \Omega_2$. Then, $\liminf_{t\to \infty} \frac{\log(I(t))}{t}\geq \phi(s_d+\epsilon)>0$ with positive probability. But that would contradict with that $\frac{\log(I(t))}{t}<0$ for all $t$ since $I(t)\leq 1$. Thus, part (1) follows.

Again, since
\begin{eqnarray*}
\dfrac{d (S+I)}{dt}&=&\mu(1-x)-\mu(S+I)-\gamma I 
\end{eqnarray*}
then
\begin{equation*}
   S(t)+I(t)=e^{-\mu t}(S(0)+I(0))+1-e^{-\mu t}\int_0^t  \mu e^{\mu s} x(s) ds -\frac{\gamma}{\mu}e^{-\mu t}\int_0^t  \mu e^{\mu s} I(s) ds
\end{equation*}

Therefore, by the generalized L'H\^opital rule (see Appendix)
\begin{equation*}
   \limsup_{t\to \infty}S(t)\leq 1-\liminf_{t\to \infty} x(t)-(1+\frac{\gamma}{\mu}) \liminf_{t\to \infty}I(t) 
\end{equation*}
and so 
\begin{equation*}
   \liminf_{t\to \infty} I(t)\leq (1-s_d-x_*)\frac{\mu}{\mu+\gamma}
\end{equation*}
Also, by the generalized L'H\^opital rule (see Appendix)
\begin{equation*}
   \liminf_{t\to \infty}S(t) \geq 1-\limsup_{t\to \infty} x(t)-(1+\frac{\gamma}{\mu}) \limsup_{t\to \infty}I(t) 
\end{equation*}
and so 
\begin{equation*}
   \limsup_{t\to \infty} I(t)\geq (1-s_d-x^*)\frac{\mu}{\mu+\gamma}
\end{equation*}
if $x^*=\limsup_{t\to \infty} x(t)<1$. Since,
\begin{equation*}
   S(t)+I(t)\leq 1
\end{equation*}
for all $t>0$, then
\begin{equation*}
   \limsup_{t\to \infty} I(t)\leq 1-s_d.
\end{equation*}
\end{proof}

Note that $s_d \downarrow \frac{1}{R_0}$ as $\sigma_1\to 0^+$. Figure \ref{fig:fig4} (a) shows a simulation of one instance of the case where $\sigma_1^2=0$ and $R_0^s=R_0=6$, in which the proportion of infected fluctuates due to the stochastic nature of vaccine uptake. In that case, $S^*=S_*=\frac{1}{R_0}$. Figure \ref{fig:fig4} (b) illustrates the case where $\sigma_1^2>0$ and $R_0^s=4.33$ with $s_d=.168$ demarcating $S_*$ and $S^*$. In both panels of Figure \ref{fig:fig4}, $I^*$ and $I_*$ have the boundaries given by Theorem \ref{thm24} in one simulation instance. See also Figure \ref{fig:Sfig4}.

Notice that if $x^*=1$, then $x_*$ must also be equal to one, and in that case $I\to 0$ and $S\to 0$ $a.s.$ If $x^*=0$, then $1-s_d$ is a divider of the $I^*$ and $I_*$. Moreover, if $x^*<1-s_d$, then $\limsup_{t\to \infty} I(t)>0$ and so $\liminf_{t\to \infty} I(t)>0$ or otherwise $\lim_{t\to \infty} I(t)=0$ since $0$ is an absorbing state for $I$. Note that $s_d$ is always greater than or equal to $\frac{1}{R_0}$ and are equal when $\sigma_1^2=0$. We can infer that $1-s_d$ is a stochastic herd immunity threshold $HIT^s$. Interestingly, $HIT^s$ is less than or equal to the deterministic herd immunity threshold $1-\frac{1}{R_0}$.

\begin{figure}[H]
    \centering
     \subfigure[]{\includegraphics[width=8 cm]{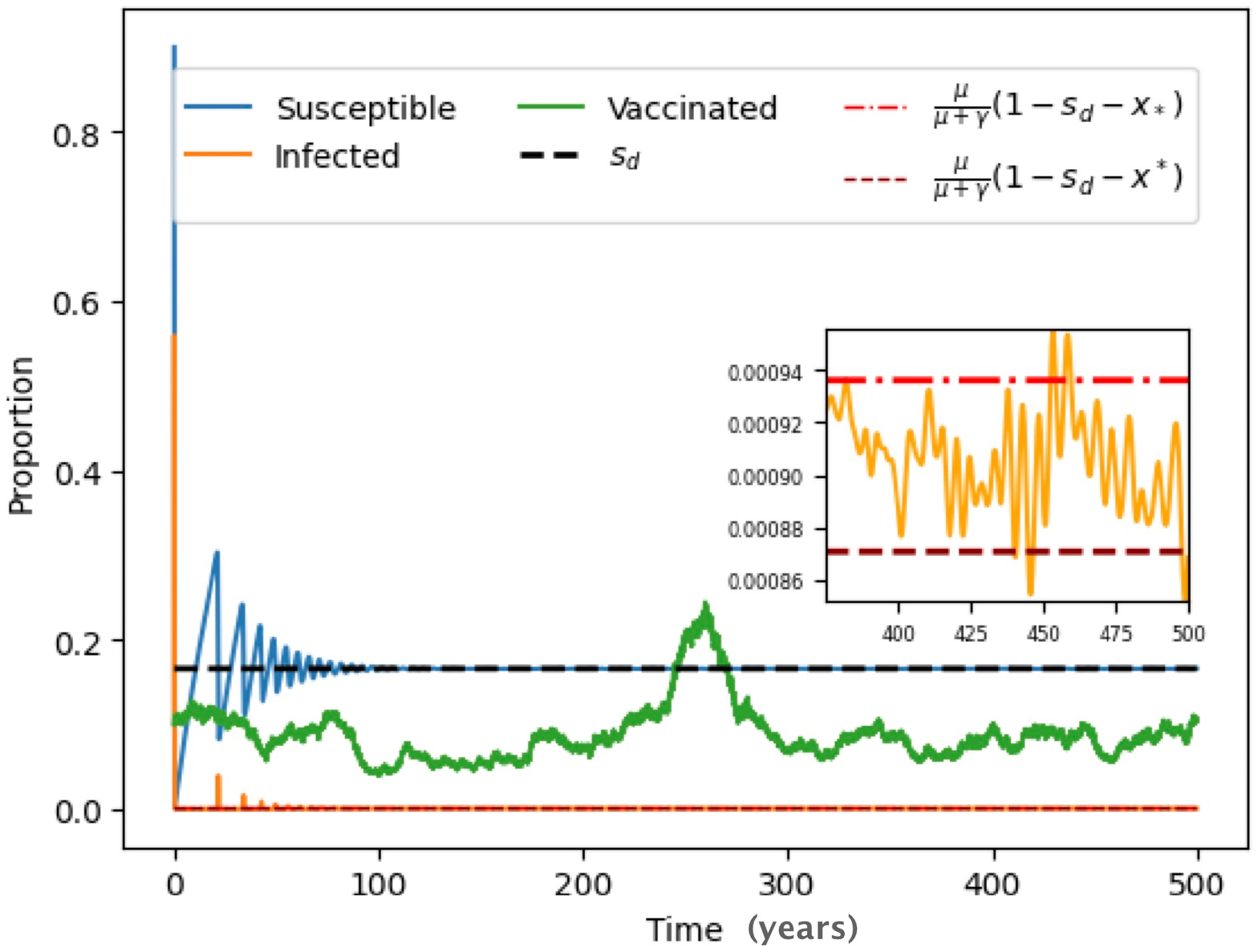}} \subfigure[]
     {\includegraphics[width=8 cm]{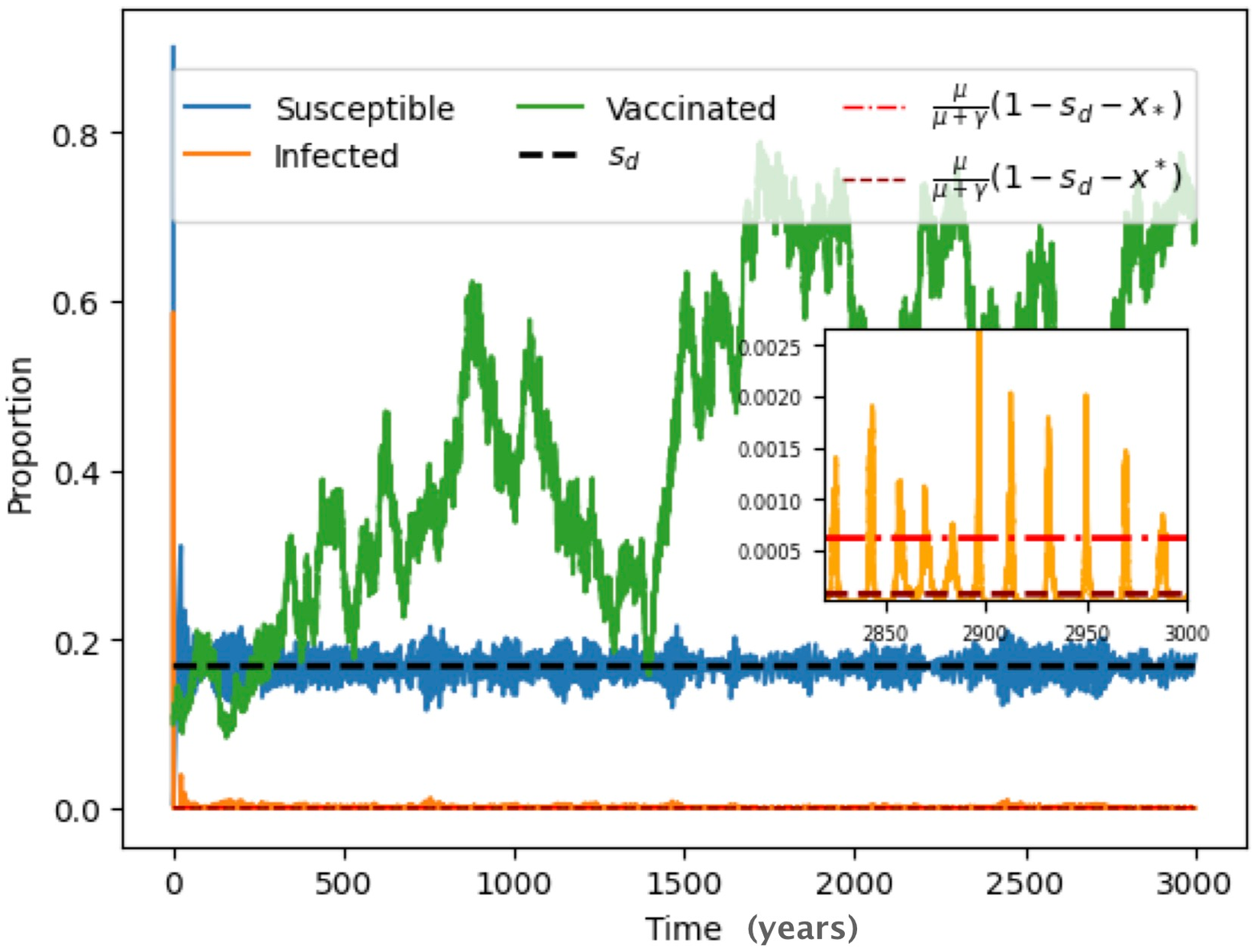}} 
    \caption{Simulation of the susceptible, infected, and vaccinated with the limits in Theorem \ref{thm24}. The disease endemic equilibrium $\mathcal{E}_5$ type is stable when (a) $\sigma_1^2=0$ in which case $R_0^s=R_0=6$ and (b) $\sigma_1^2=13$ in which case $R_0^s=4.33$. The rest of the parameter values are $\beta = 100$,  $\mu = 1/50$, $\gamma = 365/22$, $\kappa = 1.69$,  $\omega = .0004$, $\delta = 0.00005$, $\sigma_2^2=.0008$, and $\sigma_3^2=.0006$.}
    \label{fig:fig4}
\end{figure}

 Note that simulating an almost sure result requires a large number of simulation runs in which the proportion of those simulation runs revealing the result is almost one. Also, we estimated the limit supremum and infimum by long-term simulations in which the cumulative maximum and minimum values of the time-reverse of the process are found to become stable. We notice that they reach stability over a period of time but lose that stability in the last period of the simulation run, since the maximum and minimum are found over a small period of time at the end of the simulation.

The following theorem shows that the limiting temporal variability for fluctuations around levels less than or equal to the equilibria is affected by the noise in the stochastic replicator dynamics representing the parental behavior.

\begin{theorem}
    Let $(S_e,I_e,x_e)$ be an equilibrium of $\mathcal{E}_5$ type in which $x_e\neq 0,1$. If $\frac{1}{2}\frac{2 \mu + \gamma}{\beta}  \sigma_1^2 I_e < \mu$ and $\delta< \frac{1}{4} k \sigma^2$, where $\sigma^{2}=\sigma_{2}^{2}+\sigma_{3}^{2}$, then

    \begin{align*}
   & \limsup_{t\to \infty}\frac{1}{t} \int_0^t \left[ \left( S(u)- \frac{\mu}{\mu -\frac{1}{2}\frac{2 \mu + \gamma}{\beta}  \sigma_1^2 I_e} S_e \right)^2+(I(u) - I_e)^2 + (x(u)-x_e)^2 \right] du \leq  \\
    &\quad \frac{1}{m} \left[ \frac{\mu^2 S_e^{ 2}}{\mu -\frac{1}{2}\frac{2 \mu + \gamma}{\beta}  \sigma_1^2 I_e} +\frac{1}{2} \mu \kappa \sigma^2 x_e(1-x_e)+\mu x_e +\mu S_e(1-S_e) \right],
\end{align*}
almost surely, where $m=\min(\mu -\frac{1}{2}\frac{2 \mu + \gamma}{\beta}  \sigma_1^2 I_e , \, \mu +\gamma,\, \mu \left( \frac{1}{2} k \sigma^2 - 2 \delta \right))$.
\end{theorem}
\begin{proof}
Let $a = \kappa  \sigma_{3}^{2} -\delta - \omega $ and $b = \kappa  \sigma^2 -2 \delta$. Notice that we can write $a = bx_e-I_e$ and $I + a - bx = (I - I_e) - b(x-x_e)$ since $I_e,x_e$ are zeros of the drift in the stochastic replicator equation. Consider the following three non-negative functions:
\begin{align*}
V_1 &= - x_e \log \frac{x}{x_e} - (1 - x_e) \log \frac{1-x}{1-x_e}, \\
V_2 &= \frac{1}{2} \left( S - S_e + I - I_e \right)^2, \\
V_3 &= I - I_e - I_e \log \frac{I}{I_e}. 
\end{align*}
(Note: $V_1$ is the relative entropy.) Thus, It\^o formula (see Lemma \ref{ito} in the Appendix) implies
\begin{align*}
dV_1 &= \left[ \kappa \left( x - x_e \right) \left( I + a - bx \right) \right]  + \frac{1}{2} k^2 \sigma^2  (x-x_e)^2 + \frac{1}{2} k^2 \sigma^2 x_e (1-x_e)dt + k\sigma(x-x_e) dW_2 \\
&= \left[ k \left( x - x_e \right) \left( I - I_e \right) + k \left( \frac{1}{2} k \sigma^2 - b \right) \left( x - x_e \right)^2 + \frac{1}{2} k^2 \sigma^2 x_e (1-x_e) \right] dt + k \sigma \left( x - x_e \right) dW_2 
\end{align*}
Moreover,
$$dV_2 = (S - S_e + I - I_e) [\mu (1 - x) - \mu S  - (\mu + \gamma) I] dt $$
but since $\mu = \mu x_e + \mu S_e  + (\mu + \gamma) I_e$ where $S_e,I_e,x_e$ are zeros of drift of the first equation in \eqref{SIR}, then
\begin{align*}
dV_2&= \left( (S - S_e) + (I - I_e) \right) [- \mu (x - x_e) - \mu (S - S_e)  - (\mu + \gamma)(I - I_e)] \\
& =-\mu (S - S_e)^2 -(\mu + \gamma) (I - I_e)^2 -\mu (S - S_e)(x - x_e) -(2\mu + \gamma) (I - I_e)(S - S_e)  -\mu (x - x_e) (I - I_e).
\end{align*}
Finally,
$$
dV_3 = \left[\beta (I - I_e) (S - S_e) + \frac{1}{2} \sigma_1^2 I_e S^2 \right] dt  + \sigma_1 S (I - I_e) dW_1 
$$
Now, let us introduce the non-negative function $$V = \frac{\mu}{k} V_1 + V_2 + \frac{2 \mu + \gamma}{\beta} V_3.$$
Thus,
\begin{align*}
dV &= \left[ - \mu (S - S_e)^2 -(\mu +\gamma)(I - I_e)^2 -\mu \left( \frac{1}{2} k \sigma^2 - 2 \delta \right) (x-x_e)^2 - \mu  (S - S_e)(x -x_e) \right. \\
&\quad \left. +\frac{1}{2}\frac{2 \mu + \gamma}{\beta}  \sigma_1^2 I_e S^2+ \frac{1}{2} \mu \kappa \sigma^2 x_e(1-x_e) \right] dt  + \frac{2 \mu + \gamma}{\beta} \sigma_1 S (I - I_e) dW_1 + \mu \sigma (x - x_e) dW_2.
\end{align*}
But, since $- \mu  (S - S_e)(x -x_e)\leq \mu S_e x + \mu S x_e\leq \mu S_e  + \mu  x_e$, while 
$$- \mu (S - S_e)^2+\frac{1}{2}\frac{2 \mu + \gamma}{\beta}  \sigma_1^2 I_e S^2 =-(\mu -\frac{1}{2}\frac{2 \mu + \gamma}{\beta}  \sigma_1^2 I_e )\left[ S- \frac{\mu}{\mu -\frac{1}{2}\frac{2 \mu + \gamma}{\beta}  \sigma_1^2 I_e} S_e \right]^2+\frac{\mu^2 S_e^{2}}{\mu -\frac{1}{2}\frac{2 \mu + \gamma}{\beta}  \sigma_1^2 I_e}-\mu S_e^{2},$$ then
\begin{align*}
    LV &\leq  -(\mu +\gamma)(I - I_e)^2 -\mu \left( \frac{1}{2} k \sigma^2 - 2 \delta \right) (x-x_e)^2 \\
    &\quad -(\mu -\frac{1}{2}\frac{2 \mu + \gamma}{\beta}  \sigma_1^2 I_e )\left[ S- \frac{\mu}{\mu -\frac{1}{2}\frac{2 \mu + \gamma}{\beta}  \sigma_1^2 I_e} S_e \right]^2+\frac{\mu^2 S_e^{2}}{\mu -\frac{1}{2}\frac{2 \mu + \gamma}{\beta}  \sigma_1^2 I_e} \\
    &\quad +\frac{1}{2} \mu \kappa \sigma^2 x_e(1-x_e)+\mu x_e +\mu S_e(1-S_e)
\end{align*}

Now, if $dV\leq (f(S_e,I_e,x_e)-g(S,I,x))dt+\sigma(S,I,x)dW_t$ for some positive functions $V,f,g$, then $$0\leq V(S(t),I(t),x(t))\leq V(S(0),I(0),x(0))+ \int_0^t (f(S_e,I_e,x_e)-g(S,I,x))du+\int_0^t \sigma(S,I,x)dW_u$$ almost surely. But then $$\frac{1}{t} \int_0^t g(S,I,x) du\leq \frac{1}{t}V(S(0),I(0),x(0))+  f(S_e,I_e,x_e)+\frac{1}{t} \int_0^t \sigma(S,I,x)dW_u.$$
Thus,
$$\limsup_{t \to \infty}\frac{1}{t} \int_0^t g(S,I,x) du  \leq   f(S_e,I_e,x_e)$$ almost surely and the result follows.
\end{proof}

If $\sigma_1^2=0$, as in Figure \ref{fig:fig4} (a), then
\begin{equation*}
   \limsup_{t\to \infty}\frac{1}{t} \int_0^t \left[ \left( S(u)-  S_e \right)^2+(I(u) - I_e)^2 + (x(u)-x_e)^2 \right] du \leq  
    \quad \frac{1}{\min(1 ,\, \frac{1}{2} k \sigma^2 - 2 \delta )} \left[  \frac{1}{2} \kappa \sigma^2 x_e(1-x_e)+ x_e + S_e \right],
\end{equation*}
almost surely.

\section{Discussion and Conclusion}\label{Sec:Disc}

 In this paper, we introduced a novel model of stochastic replicator dynamics (RD), from stochastic game theory, to depict human choice with noise in their perceived utility function. Including white noise in the utility function serves as a way to model irrationality and stochastic deviations from the classical rational utility. Introducing the white noise to the perceived utilities of both groups gives rise to a mutation term in the drift (deterministic) component in the stochastic RD. The mutation term depicts irrational human choices in which they explore other, possibly non-optimal, behaviors. In addition, the stochastic component of the RD reflects the unpredictable nature of human behavior and the various biases and constraints that influence decision-making. Thus, adding stochastic deviations from the classical rational utility aligns more closely with the bounded rationality framework. To our knowledge, this is the first time that parental choice of vaccination for their children has been modeled using stochastic replicator dynamics to describe their bounded rationality. See also \cite{oraby3,oraby2,oraby4} for other deterministic models of bounded rationality.

We also included noise in disease transmission to investigate its effect on vaccination decision-making. Although the latter has been investigated in the literature, this is the first time that the coupled system of stochastic differential equations of SIR and RD has been studied. The coupling takes place as the stochastic process of the proportion of infected children appears in the utility functions and as the stochastic process of the vaccine uptake appears in the first equation of the stochastic SIR. The model shows new dynamics that generalize to those investigated in \cite{oraby1}. 

The presence of stochastic noise in the perceived utilities has no effect on the disease transmission if the stochastic basic reproduction number is less than one and the noise in the disease transmission is small enough. But when those conditions change, stochastic noise in perceived utility can affect the dynamical behavior of the disease due to the difference in the magnitude of that noise in the vaccinator and non-vaccinator parents in a way that can counter peer pressure. The larger the magnitude of that noise in the vaccintor's perceived utility, the more the likelihood that parents will not vaccinate their children, and vice versa for non-vaccinators. A clear and consistent view of the utility of vaccinators is required to achieve disease eradication when the disease has a large stochastic basic reproduction number. Disease eradication becomes uncertain when both magnitudes of noises approach the same value. It also depends on the initial acceptance of the vaccine.  

We noticed here the emergence of new stochastic herd immunity thresholds that are less than the deterministic one. If vaccination uptake is lower than those new thresholds, then the disease will persist. 

The effect of stochasticity on vaccine acceptance and disease eradication could be controlled through a stochastic optimal control. A discount in the cost of vaccination, including but not limited to ease of access to a safe vaccination as well as lowering its price, can achieve full coverage of the vaccine. That is true even when epidemiological, social, and cognitive parameters induce disease persistence. It is shown here that an optimal discount is not required for the full period and could be gradually eased up. The optimal control policy takes into account that vaccination efforts have practical and resource-related bounds such as finite resources and supply.   

Although the underlying model attempts to capture the population-level effects of shared or systematic uncertainties, it does not account for heterogeneity in the individual’s perceptions of payoffs. The assumption of shared or systematic stochasticity aligns with situations in which systemic external influences, e.g. environmental, cultural, informational, and economical, dominate the individual differences, especially in tightly connected or homogeneous populations. It may be less applicable to diverse populations with significant individual variability, which emphasizes the need for models that balance collective and individual-level influences on disease spread. Future work in stochastic modeling of decision-making in behavioral epidemiology should extend itself to incorporating individual-specific (occasional) stochasticity, which would allow for the exploration of diverse behaviors and richer dynamics at both individual and population levels.

In our simulations, where time is measured in years, the proportions of the susceptible, infected, and vaccine acceptors sometimes take a long time to converge to equilibrium under certain parameter settings. This slow convergence reflects the slow mixing dynamics of the system. In future work, it is warranted to find or estimate parameter values that reflect the empirically observed behavior of these variables in real life.

Adopting stochastic game theory in epidemiological modeling assists in understanding how diseases spread and in forecasting and managing public responses to health crises. This approach provides a more nuanced view of vaccine uptake through human behavior and its impact on public health, establishing behavior management as an essential element of contemporary disease control strategies.

\section*{Acknowledgment}
The authors express their sincere gratitude to the editor and the two referees for their invaluable insights and constructive suggestions, which significantly improved the quality and clarity of this manuscript.

\section*{Conflict of interest}
The authors declare no conflict of interest.


\begin{thebibliography}{10}

\bibitem{allen1994some}
{\sc Allen, L.~J.}
\newblock Some discrete-time si, sir, and sis epidemic models.
\newblock {\em Mathematical Biosciences 124\/} (1994), 83--114.

\bibitem{allen2008introduction}
{\sc Allen, L.~J.}
\newblock {\em An introduction to stochastic epidemic models}.
\newblock Springer, 2008.

\bibitem{Babaei2021}
{\sc Babaei, A., Jafari, H., Banihashemi, S., and Ahmadi, M.}
\newblock Mathematical analysis of a stochastic model for spread of coronavirus.
\newblock {\em Chaos, Solitons \& Fractals 145\/} (2021).

\bibitem{bailey1975mathematical}
{\sc Bailey, N.~T.~J.}
\newblock {\em The Mathematical Theory of Infectious Diseases and Its Applications}, 2nd~ed.
\newblock Charles Griffin \& Company, London, 1975.

\bibitem{bauch2005}
{\sc Bauch, C.~T.}
\newblock Imitation dynamics predict vaccinating behaviour.
\newblock {\em Proceedings of the Royal Society of London B: Biological Sciences 272\/} (2005), 1669--1675.

\bibitem{bauch2012}
{\sc Bauch, C.~T., and Bhattacharyya, S.}
\newblock Evolutionary game theory and social learning can determine how vaccine scares unfold.
\newblock {\em PLoS Computational Biology 8\/} (2012), e1002452.

\bibitem{bauch2004vaccination}
{\sc Bauch, C.~T., and Earn, D.~J.}
\newblock Vaccination and the theory of games.
\newblock {\em Proceedings of the National Academy of Sciences of the United States of America 101\/} (2004), 13391--13394.

\bibitem{benaim2008robust}
{\sc Bena{\"\i}m, M., Hofbauer, J., and Sandholm, W.~H.}
\newblock Robust permanence and impermanence for stochastic replicator dynamics.
\newblock {\em Journal of Biological Dynamics 2\/} (2008), 180--195.

\bibitem{betsch2013inviting}
{\sc Betsch, C., B{\"o}hm, R., and Korn, L.}
\newblock Inviting free-riders or appealing to prosocial behavior? game-theoretical reflections on communicating herd immunity in vaccine advocacy.
\newblock {\em Health Psychology 32\/} (2013), 978.

\bibitem{bish2010demographic}
{\sc Bish, A., and Michie, S.}
\newblock Demographic and attitudinal determinants of protective behaviours during a pandemic: A review.
\newblock {\em British Journal of Health Psychology 15\/} (2010), 797--824.

\bibitem{Camerer2003}
{\sc Camerer, C.~F.}
\newblock {\em Behavioral Game Theory: Experiments in Strategic Interaction}.
\newblock Princeton University Press, 2011.

\bibitem{Cao2017}
{\sc Cao, B., Shan, M., Zhang, Q., and Wang, W.}
\newblock A stochastic sis epidemic model with vaccination.
\newblock {\em Physica A: Statistical Mechanics and its Applications 486\/} (2017), 127--143.

\bibitem{donofrio2011}
{\sc D'Onofrio, A., Manfredi, P., and Poletti, P.}
\newblock The impact of vaccine side effects on the natural history of immunization programmes: an imitation-game approach.
\newblock {\em Journal of Theoretical Biology 273\/} (2011), 63--71.

\bibitem{eames2009}
{\sc Eames, K.~T.}
\newblock Networks of influence and infection: parental choices and childhood disease.
\newblock {\em Journal of the Royal Society Interface 6\/} (2009), 811--814.

\bibitem{Fishburn1982}
{\sc Fishburn, P.~C.}
\newblock {\em The Foundations of Expected Utility}.
\newblock D. Reidel Publishing Company, 1982.

\bibitem{Fudenberg1992}
{\sc Fudenberg, D., and Harris, C.}
\newblock Evolutionary dynamics with aggregate shocks.
\newblock {\em Journal of Economic Theory 57\/} (1992), 420--441.

\bibitem{fudenberg1998theory}
{\sc Fudenberg, D., and Levine, D.~K.}
\newblock {\em Theory of Learning in Games}.
\newblock MIT Press, 1998.

\bibitem{Gao2019}
{\sc Gao, N., Song, Y., Wang, X., and Liu, J.}
\newblock Dynamics of a stochastic sis epidemic model with nonlinear incidence rates.
\newblock {\em Advances in Difference Equations 41\/} (2019).

\bibitem{gray2011stochastic}
{\sc Gray, A., Greenhalgh, D., Hu, L., Mao, X., and Pan, J.}
\newblock A stochastic differential equation {SIS} epidemic model.
\newblock {\em SIAM Journal on Applied Mathematics 71\/} (2011), 876--902.

\bibitem{han2015threshold}
{\sc Han, Q., Jiang, D., Lin, S., and Yuan, C.}
\newblock The threshold of stochastic sis epidemic model with saturated incidence rate.
\newblock {\em Advances in Difference Equations 2015\/} (2015), 1--10.

\bibitem{hofbauerevolutionary}
{\sc Hofbauer, J., and Sigmund, K.}
\newblock {\em Evolutionary Games and Population Dynamics}.
\newblock Cambridge University Press, 1998.

\bibitem{ibuka2014}
{\sc Ibuka, Y., Li, M., Vietri, J., Chapman, G.~B., and Galvani, A.~P.}
\newblock Free-riding behavior in vaccination decisions: an experimental study.
\newblock {\em PLoS ONE 9\/} (2014), e87164.

\bibitem{hindawi2023cholera}
{\sc Iddrisu, W.~A., Iddrisu, I., and Iddrisu, A.-K.}
\newblock Modeling cholera epidemiology using stochastic differential equations.
\newblock {\em Journal of Applied Mathematics 2023\/} (2023), 7232395.

\bibitem{Ji2014}
{\sc Ji, C., and Jiang, D.}
\newblock Threshold behaviour of a stochastic {SIR} model.
\newblock {\em Applied Mathematical Modelling 38\/} (2014), 5067--5079.

\bibitem{Jian2024}
{\sc Jian, L., Bai, X., and Ma, S.}
\newblock Stochastic dynamical analysis for the complex infectious disease model driven by multisource noises.
\newblock {\em PLOS ONE 19\/} (2024), e0296183.

\bibitem{kermack1927contribution}
{\sc Kermack, W.~O., and McKendrick, A.~G.}
\newblock A contribution to the mathematical theory of epidemics.
\newblock {\em Proceedings of the Royal Society of London. Series A, Containing Papers of a Mathematical and Physical Character 115\/} (1927), 700--721.

\bibitem{kloeden1992}
{\sc Kloeden, P.~E., and Platen, E.}
\newblock {\em Numerical Solution of Stochastic Differential Equations}.
\newblock Springer, Berlin, Heidelberg, 1992.

\bibitem{Lahrouz2014}
{\sc Lahrouz, A., and Settati, A.}
\newblock Necessary and sufficient condition for extinction and persistence of sirs system with random perturbation.
\newblock {\em Applied Mathematics and Computation 233\/} (2014), 10--19.

\bibitem{Liu2023}
{\sc Liu, Y.}
\newblock Stationary distribution and probability density function of a stochastic waterborne pathogen model with saturated direct and indirect transmissions.
\newblock {\em Mathematical Methods in the Applied Sciences 46\/} (2023), 13830--13854.

\bibitem{manfredi2013}
{\sc Manfredi, P., and D'Onofrio, A.}
\newblock {\em Modeling the interplay between human behavior and the spread of infectious diseases}.
\newblock Springer, 2013.

\bibitem{mao2007stochastic}
{\sc Mao, X.}
\newblock {\em Stochastic Differential Equations and Applications}.
\newblock Woodhead Publishing, 2007.

\bibitem{ndeffo2012}
{\sc Ndeffo~Mbah, M.~L., Liu, J., Bauch, C.~T., Tekel, Y., Medlock, J., and Meyers, L.~A., e.~a.}
\newblock The impact of imitation on vaccination behavior in social contact networks.
\newblock {\em PLoS Computational Biology 8\/} (2012), e1002469.

\bibitem{oraby3}
{\sc Oraby, T., and Balogh, A.}
\newblock Modeling the effect of observational social learning on parental decision-making for childhood vaccination and diseases spread over household networks.
\newblock {\em Frontiers in Epidemiology 3\/} (2024), 1177752.

\bibitem{oraby2}
{\sc Oraby, T., and Bauch, C.~T.}
\newblock Bounded rationality alters the dynamics of paediatric immunization acceptance.
\newblock {\em Scientific Reports 5\/} (2015), 10724.

\bibitem{oraby1}
{\sc Oraby, T., Thampi, V., and Bauch, C.~T.}
\newblock The influence of social norms on the dynamics of vaccinating behaviour for paediatric infectious diseases.
\newblock {\em Proceedings of the Royal Society B: Biological Sciences 281\/}.

\bibitem{reluga2006}
{\sc Reluga, T.~C., Bauch, C.~T., and Galvani, A.~P.}
\newblock Evolving public perceptions and stability in vaccine uptake.
\newblock {\em Mathematical Biosciences 204\/} (2006), 185--198.

\bibitem{Savage1954}
{\sc Savage, L.~J.}
\newblock {\em The Foundations of Statistics}.
\newblock John Wiley and Sons, 1954.

\bibitem{Simon1982}
{\sc Simon, H.~A.}
\newblock {\em Models of Bounded Rationality: Empirically Grounded Economic Reason}.
\newblock MIT Press, 1982.

\bibitem{Simon1984}
{\sc Simon, H.~A.}
\newblock {\em Models of Bounded Rationality: Behavioral Economics and Business Organization}.
\newblock MIT Press, 1984.

\bibitem{tornatore2005stability}
{\sc Tornatore, E., Buccellato, S.~M., and Vetro, P.}
\newblock Stability of a stochastic sir system.
\newblock {\em Physica A: Statistical Mechanics and its Applications 354\/} (2005), 111--126.

\bibitem{tversky1974judgment}
{\sc Tversky, A., and Kahneman, D.}
\newblock Judgment under uncertainty: Heuristics and biases.
\newblock {\em Science 185\/} (1974), 1124--1131.

\bibitem{Von1944}
{\sc Von~Neumann, J., and Morgenstern, O.}
\newblock {\em Theory of Games and Economic Behavior}.
\newblock Princeton University Press, 1944.

\bibitem{wagner2020}
{\sc Wagner, C.~E., Prentice, J.~A., Saad-Roy, C.~M., Yang, L., Grenfell, B.~T., and Levin, S., e.~a.}
\newblock Economic and behavioral influencers of vaccination and antimicrobial use.
\newblock {\em Frontiers in Public Health 8\/} (2020), 614113.

\bibitem{wang2016}
{\sc Wang, Z., Bauch, C.~T., Bhattacharyya, S., D'Onofrio, A., Manfredi, P., and Perc, M., e.~a.}
\newblock Statistical physics of vaccination.
\newblock {\em Physics Reports 664\/} (2016), 1--113.

\bibitem{oraby4}
{\sc Yin, W., Ndeffo-Mbah, M., and Oraby, T.}
\newblock Vaccination games of boundedly rational parents toward new childhood immunization.
\newblock {\em Preprint at Research Square\/}.

\bibitem{yong2012stochastic}
{\sc Yong, J., and Zhou, X.~Y.}
\newblock {\em Stochastic Controls: Hamiltonian Systems and HJB Equations}, vol.~43.
\newblock Springer Science \& Business Media, 2012.

\bibitem{Young}
{\sc Young, W.~H.}
\newblock On indeterminate forms.
\newblock {\em Proceedings of the London Mathematical Society s2-8\/} (1910), 40--76.

\end{thebibliography}

\newpage
\appendix
\section{Supporting Material.}
In the appendix, we will present some important definitions, lemmas, and propositions, some of which can be found in the literature. The rest are proven here.

\subsection{Definitions and Supporting Lemmas.}\label{A.supp}
\begin{lemma}\label{lHop} The generalized L'H\^opital rule  \cite{Young} 
\begin{equation} \label{gen_LHospital}
    \liminf_{t\to \infty} \dfrac{f'(t)}{g'(t)}\leq \liminf_{t\to \infty} \dfrac{f(t)}{g(t)}\leq\limsup_{t\to \infty} \dfrac{f(t)}{g(t)}\leq\limsup_{t\to \infty} \dfrac{f'(t)}{g'(t)} 
\end{equation}    
\end{lemma}

Let $(\Omega,\mathcal{F},\{\mathcal{F}_t\}_{t\geq 0},\mathbb{P})$ be a complete probability space with filtration $\{\mathcal{F}_t\}_{t\geq 0}$ with respect to which all of the stochastic processes introduced in the following are defined.

\begin{lemma}[It\^o's formula]\label{ito}
Let  the one-dimensional stochastic differential equation
$$
    dX({t})=\mu_{t}dt+\sigma_{t}dW({t})
$$
       If $f\left(t,x\right)$ is a twice-differentiable scalar function, then
$$
    df(t,X(t))=\left(\frac{\partial f}{\partial t}+\mu_{t}\frac{\partial f}{\partial x}+\frac{\sigma_{t}^{2}}{2}\frac{\partial^{2}f}{\partial x^{2}}\right)dt+\sigma_{t}\frac{\partial f}{\partial x}dW(t)
$$

\end{lemma}

A d-dimensional stochastic differential equation
$$dX(t)=F(t,X(t)) dt + G(t,X(t)) dW(t)$$
where the function $F(t,x)=\left(f_{i}(t,x)\right)_{1\leq i\leq d}$ is a $d\times 1$ vector defined on $[0,\infty]\times \mathbb{R}^d$ and $G(t,x)=\left(g_{i,j}(t,x)\right)_{1\leq i\leq d,1\leq j\leq n}$ is a $d\times n$ vector defined on $[0,\infty]\times \mathbb{R}^d$ are both locally Lipschitz functions in $x$. Let $W(t)$ be an $n-$ dimensional Brownian motion. Let $V$ be a function in $C^{1,2}([0,\infty]\times \mathbb{R}^d,[0,\infty))$ and let the differential operator $L$ be defined by
$$LV(t,x)=\dfrac{\partial V(t,x)}{\partial t}+F(t,x)^T DV(t,x) +\frac12 Trace(G(t,x)^T  H(t,x) G(t,x) ) $$
where $DV(t,x)=\dfrac{\partial V(t,x)}{\partial x}$ is the gradient of $V(t,x)$ and $H(t,x)=\dfrac{\partial^2 V(t,x)}{\partial x^2}$ is the Hessian matrix of $V(t,x)$. It\^o's formula for $X(t)$ and such function $V$ states that
$$dV(t,X(t))=LV(t,X(t))dt+DV(t,X(t))^T G(t,X(t)) dW(t)$$
for $t\geq 0$.

The following lemma can be found in \cite{mao2007stochastic}.

\begin{lemma}\label{martin}
Let $\{M_t:t\geq 0\}$ be a real-valued continuous local martingale with respect to $\{\mathcal{F}_t\}_{t\geq 0}$ such that $M_0=0$. If $\limsup_{t \to \infty} \frac{\mathbb{E}(M_t^2)}{t}   < \infty$ then $\lim_{t \to \infty} \frac{M_t}{t}=0$ $a.s.$
\end{lemma}

The following lemma is an extension to the lemmas that were introduced in the appendix of \cite{Ji2014}. It is new to our knowledge. 

\begin{lemma}\label{Ap1}
Consider a function $f\in C([0,\infty)\times \Omega,(0,\infty))$. Let $a(t)$ be a sequence such that $\lim_{t \to \infty} \dfrac{a(t)}{t}=a$ a.s. and let $a$ and $b$ be two positive constants such that
$$\log(f(t))\geq (\leq) \, a(t)-b\int_0^t f(s)ds $$
for all $t\geq 0$ a.s. Then
$$\liminf_{t\to \infty} (\limsup_{t\to \infty})\, \frac{1}{t} \int_0^t f(s)ds\geq (\leq) \, \dfrac{a}{b} \quad a.s.$$
\end{lemma}
\begin{proof} We will prove the "$\geq$" part and the "$\leq$" follows using the same approach. Let $g(t):=b\,f(t)$ and so the inequality becomes
$$\log(g(t))\geq (\leq) \, a(t)+\log(b)-\int_0^t g(s)ds. $$

Let $\lim_{t \to \infty} \dfrac{a(t)}{t}=a$ a.s. hence for arbitrary $\epsilon>0$ but small enough so that there exists $\Omega_\epsilon\in \mathcal{F}$ with $\mathbb{P}(\Omega_\epsilon)\geq 1-\epsilon$ and $0<a-\epsilon \leq \dfrac{a(t)}{t}(\omega)\leq a+\epsilon$ for all $t\geq t_0(\omega)$ for each $\omega \in \Omega_\epsilon$.  Let $G(t)=\int_0^t g(s) ds$ for all $t\geq 0$, then $dG(t)=g(t)dt$ and $$G(t)\geq a(t)+\log(b) -\log(g(t))$$  for all $t\geq 0$. Therefore, 
$$de^{G(t)}=e^{G(t)}g(t)dt\geq e^{a(t)+\log(b) -\log(g(t))} g(t)dt=b e^{a(t)} dt$$ and $e^{a(t)}\geq e^{(a-\epsilon)t}$ on $\Omega_\epsilon$. Thus, for $t>t_0$
$$G(t)\geq \log(e^{G(t_0)}+\frac{b}{(a-\epsilon)}[e^{(a-\epsilon)t}-e^{(a-\epsilon)t_0}])$$ Hence, by L'H\^opital rule (see Lemma \ref{lHop})
$$\liminf_{n\to \infty} \dfrac{G(t)}{t}\geq a-\epsilon$$
and so
$$\liminf_{t\to \infty} \frac{1}{t} \int_0^t f(s)ds\geq \dfrac{a-\epsilon}{b} \quad \text{ on }\Omega_\epsilon$$
for arbitrary $\epsilon>0$, thus $$\liminf_{t\to \infty}  \frac{1}{t} \int_0^t f(s)ds\geq  \, \dfrac{a}{b} \quad a.s.$$
\end{proof}

\subsection{Existence, Uniqueness, and Boundedness.}\label{A.ext}

Here, we will prove the positive invariance of the solution set,
$\mathcal{S}=\{(s,i,x)\in \mathbb{R}_+^3: 0\leq s+i \leq 1 \text{ and } 0\leq x \leq 1\}$. It is not straightforward since the drift and volatility in the stochastic equations are locally Lipschitz continuous but not growing linearly. That might cause a blow-up in the solutions of the stochastic differential equations.
\begin{lemma}\label{invar}
If model \eqref{SIR} has initial state $(S(0),I(0),x(0))\in \mathcal{S}$, then the model has a unique solution
$(S(t),I(t),x(t))\in \mathcal{S}$ for all $t\geq 0$ with probability one.
\end{lemma}
\begin{proof}
All of the relationships in this theorem are true point-wise. By construction of the stochastic replicator equation $0\leq x(t)\leq 1$ for all $t>0$ if $0\leq x(0)\leq 1$. Also, $d(S+I+R)=0$, thus if $0\leq S(0)+I(0)\leq 1$ and 
\begin{eqnarray*}
\dfrac{d (S+I)(t)}{dt}&=&\mu(1-x(t))-\mu(S+I)(t)-\gamma I(t) \\
&\leq& \mu(1-x(t))-\mu(S+I)(t) \\
&\leq& \mu -\mu(S+I)(t)
\end{eqnarray*}
for all $t>0$. Thus, for any $t>0$,
$$
S(t)+I(t)\leq (S(0)+I(0))\, e^{-\mu t}+1-e^{-\mu t}\leq 1.
$$
On the other hand, 
\begin{align*}
\dfrac{d (S+I)(t)}{dt}&=\mu(1-x(t))-\mu(S+I)(t)-\gamma I(t) \\
&\geq \mu(1-x(t))-(\mu+\gamma)(S+I)(t) \\
&\geq -(\mu+\gamma)(S+I)(t)
\end{align*}
for all $t>0$. Thus, for any $t>0$,
$$
S(t)+I(t)\geq (S(0)+I(0))\, e^{-(\mu+\gamma) t}\geq 0.
$$
\end{proof}

The processes $S$, $I$, and $x$ are positive bounded processes. Moreover, they have bounded limits are shown in the following proposition. 
\begin{proposition}\label{prop1}
If the initial state of the model \eqref{SIR} is such that $(S(0),I(0),x(0))\in \mathcal{S}$, then the following inequality always holds true.    
 \begin{equation}\label{liminfsup_inequality}
     1-x^* - \frac{\gamma}{\mu} I^* \leq S_*+I_* \leq  S^* + I^*\leq 1-x_* - \frac{\gamma}{\mu} I_* \quad a.s.
 \end{equation}

\end{proposition}
\begin{proof}
Since
\begin{equation} \label{ode:S+I}
\frac{d (S+I)(t)}{dt}=\mu(1-x(t))-\mu(S+I)(t)-\gamma I(t) 
\end{equation}
then
\begin{equation}
   S(t)+I(t)=e^{-\mu t}(S(0)+I(0))+1-e^{-\mu t}\int_0^t  \mu e^{\mu s} x(s) ds -\frac{\gamma}{\mu}e^{-\mu t}\int_0^t  \mu e^{\mu s} I(s) ds.
\end{equation} 

\noindent Inequality \eqref{liminfsup_inequality} follows from using Lemma \ref{invar} and the generalized L'H\^opital rule (see inequality \eqref{gen_LHospital} in the Appendix) applied to the terms $\displaystyle e^{-\mu t}\int_{0}^{t} e^{\mu s}x(s)ds$ and $\displaystyle e^{-\mu t}\int_{0}^{t} e^{\mu s}I(s)ds$.

\end{proof}

\subsection{More simulations.}\label{more_sim}

\begin{figure}[H]
    \centering
     \subfigure[]{\includegraphics[width=8 cm]{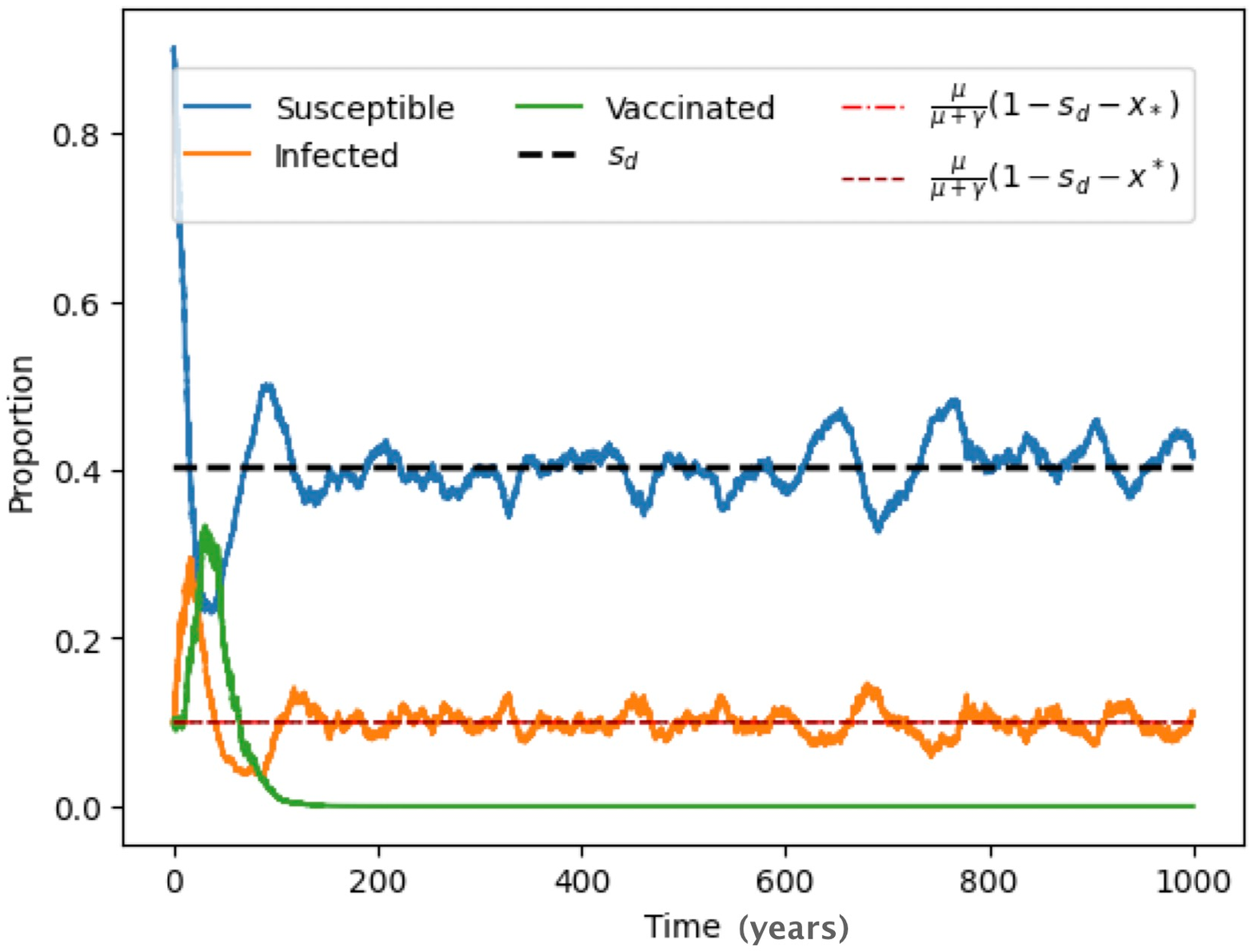}} 
     \subfigure[]{\includegraphics[width=8 cm]{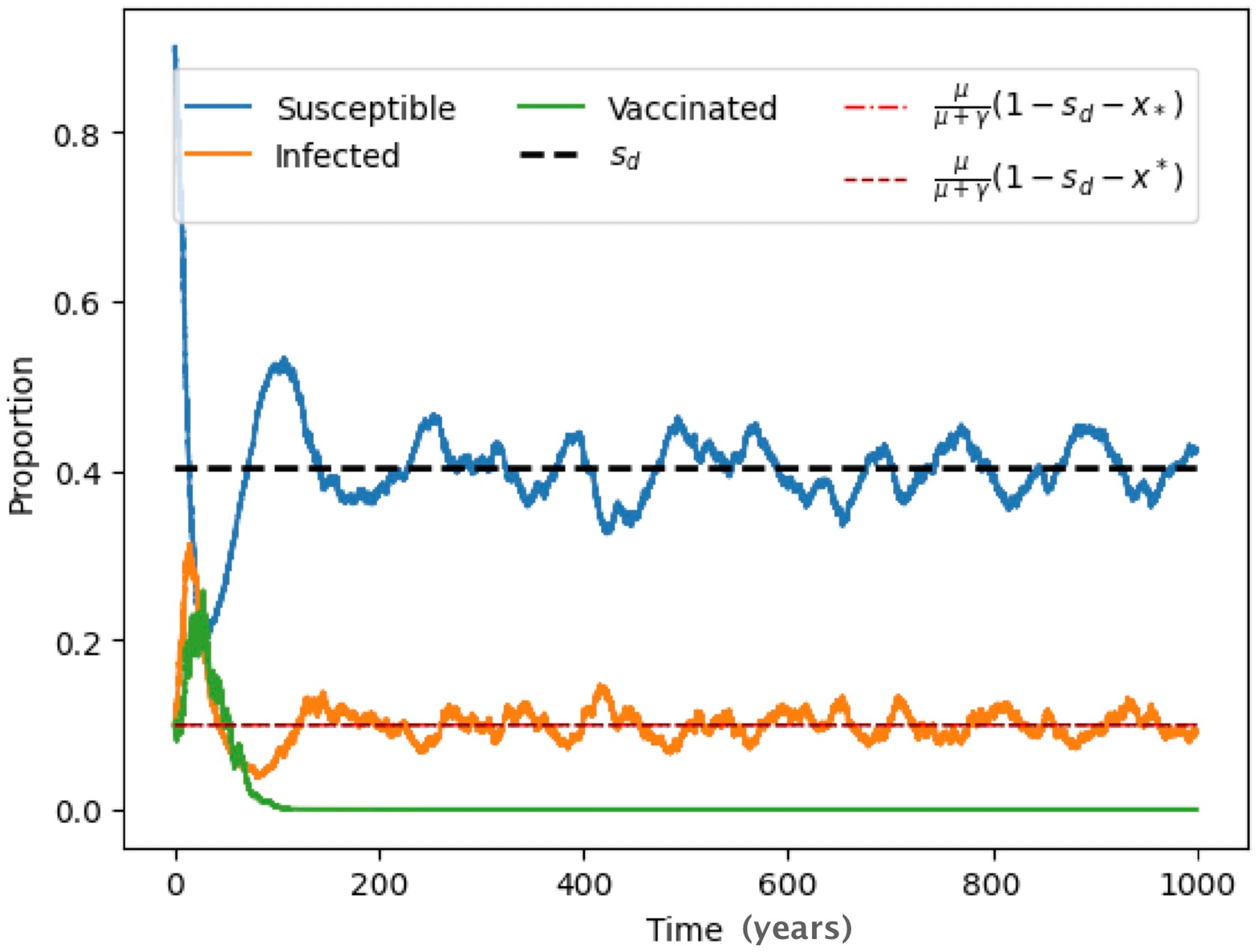}}
     \subfigure[]{\includegraphics[width=8 cm]{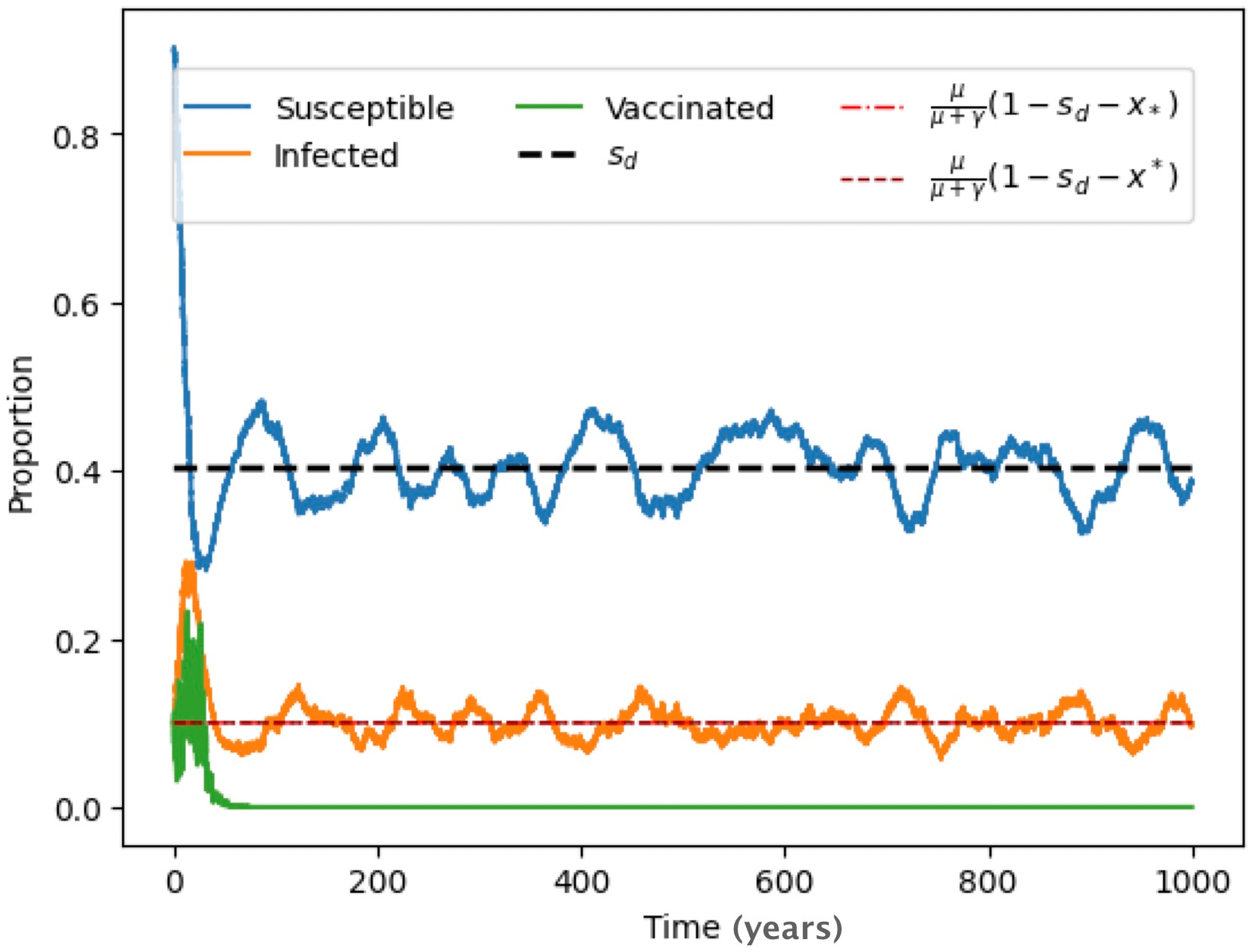}}
    \caption{Simulation of the susceptible, infected, and vaccinated with the limits in Theorem \ref{thm24} with $\sigma_2^2=.01$ in (a), $\sigma_2^2=.04$ in (b), and $\sigma_2^2=1$ in (c). The disease endemic equilibrium $\mathcal{E}_4$ type is stable in which case $R_0^s=2.4$. The rest of the parameter values are $\beta = .3$,  $\mu = 1/50$, $\gamma = .1$, $\kappa = .5$,  $\omega = .1$, $\delta = .1$, $\sigma_1^2=.01$, and $\sigma_3^2=.01$.}
    \label{fig:Sfig4}
\end{figure}

\subsection{Stochastic Optimal Control Problem to Achieve Vaccine Acceptance}\label{optimal}
For convenience, we define the vector $
y(t) = [S(t), I(t), x(t)]^\top$. The system can be re-written as the SDE vector form 
$$dy(t) = f(y(t), u(t)) dt + g(y(t)) dW(t),
$$ with initial condition $y(0) = [S(0), I(0), x(0)]^\top = y_0$ where the functions $ f $ and $ g $ are vectors with the following components
\begin{equation*}
    \begin{aligned}
f_1(y(t), u(t)) &= \mu (1-x) - \beta S I - \mu S, \\
f_2(y(t), u(t)) &= \beta S I - (\mu + \gamma) I, \\
f_3(y(t), u(t)) &= \kappa x (1-x) \left[ -\omega \, (1-u) + I + \delta (2x-1) + \kappa \left( \sigma_3^2 - (\sigma_2^2 + \sigma_3^2) x \right) \right],
\end{aligned}
\end{equation*}
and
\begin{equation*}
\begin{aligned}
g_1(y(t)) &= -\sigma_1 S I, \\
g_2(y(t)) &= \sigma_1 S I, \\
g_3(y(t)) &= \kappa \sqrt{\sigma_2^2 + \sigma_3^2} \, x (1-x),
\end{aligned}
\end{equation*}
respectively. We consider the quadratic cost functional:
$$
J(u) = - \mathbb{E} \left[ \int_0^{T_f} \left( \alpha_1 S(t) + \alpha_2 I(t) + \frac{\alpha_3}{2} u^2(t) \right) dt \right].
$$
for some constants $ \alpha_1, \alpha_2,\alpha_3 \geq 0$. The goal is to find an optimal control $ u^*(t) $ such that $J(u) \leq J(u^*)$, for all $u \in U$, where the set of admissible controls is $U = \{u(t): u(t) \in [0, u_{\max}], \text{ for all } t \in (0, T_f] \}$ for some $0 \leq u_{\max}< 1$, and $T_f$ is the termination time.

To use the stochastic maximum principle, we define Hamiltonian $ H(y, u, p, q) $ by
$$
H(y, u, p, q) = -l(y, u) + \langle f(y, u), p \rangle + \langle g(y), q \rangle
$$
where $ \langle \bf{u},\bf{v} \rangle =\sum_{i=1}^n u_i v_i $ is the Euclidean inner product and $l(y, u)=\alpha_1 S + \alpha_2 I + \frac{\alpha_3}{2} u^2$, while $ p = [p_1, p_2, p_3]^\top $ is the adjoint vector of costates and $ q = [q_1, q_2, q_3]^\top $ are costates volatilities. Hence, the associated Hamiltonian is given by
\begin{equation}
\begin{aligned}
H &= -\alpha_1 S - \alpha_2 I - \frac{\alpha_3}{2} u^2 \\
&\quad + p_1 \left( \mu (1-x) - \beta S I - \mu S \right) + p_2 \left( \beta S I - (\mu + \gamma) I \right) \\
&\quad + p_3 \left( \kappa x (1-x) \left[ -\omega \, (1-u) + I + \delta (2x-1) + \kappa \left( \sigma_3^2 - (\sigma_2^2 + \sigma_3^2) x \right) \right] \right) \\
&\quad + q_1 \left( -\sigma_1 S I \right) + q_2 \left( \sigma_1 S I \right)+ q_3 \left(\kappa \sqrt{\sigma_2^2 + \sigma_3^2}\, x (1-x)\right).
\end{aligned}
\end{equation}
According to the stochastic maximum principle, 
$$H(y^*, u^*, p, q)=\max_{u\in U} H(y^*, u, p, q)$$
where the optimal states satisfy
$$
dy^*(t) = \frac{\partial H(y^*, u^*, p, q)}{\partial p} dt + g(y^*(t)) dW(t),
$$
with the backward stochastic differential equations of the costates
$$
dp(t) = -\frac{\partial H(y^*, u^*, p, q)}{\partial y} dt + q(t) dW(t),
$$
where $u^*$ solves
$$
\frac{\partial H(y^*, u, p, q)}{\partial u}=0,
$$
and exists since $H$ is concave in $u$. Thus, the optimal control is given by
\begin{equation}\label{ustar}
  u^* = \max \left\{ 0,\min \left\{  \frac{\kappa  \omega}{\alpha_3}p_3 x^* (1-x^*) ,u_{\max} \right\} \right\}.
\end{equation}
The function $p_3(t)$ is found by solving the system of backward stochastic differential equations (BSDEs) given by
\begin{align}
dp_1(t) &= \left[ \alpha_1 + p_1 \left( \beta I + \mu - {\color{black}{2}}\sigma_1^2 S I^2 \right) -  p_2 \left(\beta I+{\color{black}{2}} \sigma_1^2 S I^2\right)  \right] dt  -\sigma_1 p_1 S I dW_1(t), \\
dp_2(t) &= \left[ \alpha_2 +  p_1  (\beta S-{\color{black}{2}} \sigma_1^2 S^2 I) - p_2 \left( \beta S - (\mu + \gamma)+{\color{black}{2}} \sigma_1^2 S^2 I \right) -\kappa p_3  x (1-x)  \right] dt + \sigma_1 p_2 S I dW_1(t), \\
dp_3(t) &= \left[  \mu p_1  - \kappa p_3 \left(  \left( 1-2x \right) \left( -\omega \, (1-u) + I + \delta (2x-1) + \kappa \left( \sigma_3^2 - (\sigma_2^2 + \sigma_3^2) x \right) \right) \right.\right. \nonumber \\
&\quad \left. \left. +  \left( 2 \delta - \kappa (\sigma_2^2 + \sigma_3^2) \right)  x (1-x) +  {\color{black}{2}} \kappa (\sigma_2^2 + \sigma_3^2) x (1-x) \left( 1-2x \right) \right)\right] dt \nonumber \\ & \quad  + \kappa \sqrt{\sigma_2^2 + \sigma_3^2} p_3 x (1-x) dW_2(t),
\end{align}
such that $p_i(T_f)=0$ for $i=1,2,3$. While $x^*(t)$ solves Equation \eqref{SIR3.1} at $u^*(t)$ given by \eqref{ustar}. Including the reduction of susceptible in the objective function results in a longer reduction in the vaccination cost than when excluded from the objective.

\end{document}